\documentclass[a4paper,10pt]{amsart}
\usepackage[utf8]{inputenc}
\usepackage{amsmath}
\usepackage{amssymb}
\usepackage{amsthm}
\usepackage{enumitem}
\usepackage{tikz}
\usepackage{todonotes}
\usepackage{hyperref}
\usepackage{url}

\usetikzlibrary{decorations.markings}
\usetikzlibrary{arrows}
\usetikzlibrary{cd}

\setlist[enumerate]{label=\rm{(\arabic*)}, ref=(\arabic*)}

\DeclareMathOperator{\St}{St}
\DeclareMathOperator{\Aut}{Aut}
\DeclareMathOperator{\rist}{rist}

\DeclareMathOperator{\Sym}{Sym}

\newcommand{\FF}{\mathbb{F}}
\newcommand{\NN}{\mathbb{N}}
\newcommand{\ZZ}{\mathbb{Z}}
\newcommand{\MF}{\mathcal{MF}}

\newtheorem{thm}{Theorem}[section]
\newtheorem{lem}[thm]{Lemma}
\newtheorem{prop}[thm]{Proposition}
\newtheorem{cor}[thm]{Corollary}

\theoremstyle{definition}
\newtheorem*{qu}{Question}
\newtheorem{defn}[thm]{Definition}
\newtheorem{rem}[thm]{Remark}
\newtheorem*{notation}{Notation}
\newtheorem{example}[thm]{Example}

\title{Maximal subgroups of groups of intermediate growth}

\author[D.\ Francoeur]{Dominik Francoeur}
\address{Section de math\'ematiques\\
Universit\'e de Gen\`eve\\
2-4 rue du Li\`evre\\
Case postale 64\\
1211 Gen\`eve 4, Switzerland}
\email{dominik.francoeur@unige.ch}

\author[A.\ Garrido]{Alejandra Garrido}
\address{School of Mathematical \& Physical Sciences\\
University of Newcastle\\
University Drive\\
Callaghan NSW 2308 Australia}
\email{aga663@newcastle.edu.au}

\date{\today}



\begin{document}

\begin{abstract}
 
 Finding the number of maximal subgroups of infinite index of a finitely generated group is a natural problem that has been solved for several classes of ``geometric'' groups
 (linear groups, hyperbolic groups, mapping class groups, etc). 
 Here we provide a solution for a family of groups with a different geometric origin: groups of intermediate growth that act on rooted binary trees. 
 In particular, we show that the non-torsion iterated monodromy groups of the tent map 
 (a special case of some groups first introduced by {\v{S}}uni{\'c} in \cite{Sunic} as ``siblings of the Grigorchuk group'')
 have exactly countably many maximal subgroups of infinite index, and describe them up to conjugacy.
 This is in contrast to the torsion case (e.g.\ Grigorchuk group) where there are no maximal subgroups of infinite index. 
 It is  also in contrast to the above-mentioned geometric groups, where there are either none or uncountably many such subgroups. 
 
 Along the way we show that all the groups defined by {\v{S}}uni{\'c} have the congruence subgroup property and are just infinite. 
%
%
%
%
\end{abstract}

 \maketitle

 \section{Introduction}
 
 Maximal subgroups (proper subgroups which are maximal with respect to containment) are one of the most basic aspects of a group that one can examine.
Their study goes back to the early days of group theory, to Galois and Frattini, 
and the maximal subgroups of finite groups have been widely studied for a long time. 
The study of maximal subgroups of infinite groups is more recent and subtle, as maximal subgroups may not even exist. 
Nevertheless, in a finitely generated group every proper subgroup is contained in a maximal one 
(a fact that does not require Zorn's lemma, see \cite{BHNeumann}), so it is very natural to study the maximal subgroups of a finitely generated group. 

When one studies groups by means of their actions, the most elementary actions are those on a set, yielding permutation representations. 
In this setting, maximal subgroups also play an important role as they correspond in a natural way to the irreducible components of a permutation action: the primitive  actions. 
A subgroup $H<G$ is maximal if and only if the action of $G$ on the coset space $G/H$ is primitive. 
Recall that the action of a group $G$ on a set $X$ containing at least two elements is primitive if it is transitive and leaves no partition of $X$ invariant except for the trivial ones (the whole set and the singletons). 
Any transitive imprimitive action can be embedded in the wreath product of two simpler actions (the action of the set-wise stabilizer of an element of the partition, and the action on elements of the partition).

Let us denote by $\MF$ the class of groups all of whose maximal subgroups have finite index. 
Which finitely generated groups are in this class and which are not?
In other words, which finitely generated groups have primitive actions only on finite sets?
Which have primitive actions on infinite sets?

A famous theorem of Margulis and Soifer \cite{MargulisSoifer} states that a finitely generated linear group is in $\MF$ if and only if it is virtually solvable. 
Margulis and Soifer's result was the inspiration for \cite{GelanderGlasner}, where a characterization of the groups in $\MF$ was obtained for (countable) groups in other geometric settings:
subgroups of mapping class groups, word hyperbolic (more generally, convergence) groups, groups acting minimally on non-rooted trees.
In all these settings the groups outside $\MF$ contain free subgroups (this is a key ingredient of the proof). 

Some examples of groups outside $\MF$ without free subgroups include Thompson's group $F$ \cite{Savchuk}, Tarski monsters and some Golod--Shafarevich groups 
(those that map onto a Tarski monster). 
It should also be noted that it follows from \cite[Section 7]{ErshovJaikin} that every Golod--Shafarevich group has a $p$-torsion 
(where $p$ is a prime) quotient which is Golod--Shafarevich and is in $\MF$. 

Note that all cited examples of groups outside $\MF$ are in some sense ``big'', they are all of exponential growth. 
Since all groups of polynomial growth are in $\MF$ (by Gromov's theorem), it is reasonable to wonder what happens for groups of intermediate growth. 
A starting point is to study Grigorchuk's groups of intermediate growth, \cite{Grigorchuk_intermediate}.
This is was done by Pervova in \cite{Pervova} where it was shown that the torsion groups in this family are in $\MF$. 
Another family of groups that are similar in spirit to Grigorchuk's were defined by Gupta and Sidki in \cite{GuptaSidki}. 
The aim there was to give uncomplicated examples of finitely generated infinite $p$-groups and their growth is not yet established. 
Pervova also showed in \cite{Pervova} that these groups are in $\MF$.
This result was extended in \cite{Klopsch1,Klopsch2} to all torsion groups in a generalization of Grigorchuk--Gupta--Sidki groups, so-called multi-edge spinal groups. 
Moreover, if a group is in $\MF$ then all groups which are abstractly commensurable with it are also in $\MF$
(two groups are abstractly commensurable if they have isomorphic finite index subgroups). 
Thus, by \cite{GrigorchukWilson} and \cite{Garrido}, all infinite finitely generated subgroups of the Grigorchuk group and the Gupta--Sidki 3-group are in $\MF$. 

All groups in the above paragraph are branch groups (see Definition \ref{def:branch_def}).
Indeed, the definition of branch group is inspired by their properties. 
Pervova's result prompted the question whether every finitely generated branch group is in $\MF$ (this was problem 6.2 in \cite{Grigorchuk_solved_unsolved}). 
This was answered negatively by Bondarenko in \cite{Bondarenko}, by showing that all groups of the type that appear in  \cite{Neumann, Segal,Wilson} have a maximal subgroup of infinite index.
However, these groups are very special:  
each level stabilizer decomposes as a direct product of copies of the original group, a property which most branch groups do not enjoy.
They are also of exponential growth; in fact, the groups in \cite{Wilson} have non-uniform exponential growth. 



It is natural to wonder, and was explicitly asked in \cite[Example 3.10(8)]{Cornulier} whether there exist groups of intermediate growth outside $\MF$. 
When we started working on this project, no examples were known, but recently Nekrashevych \cite{Nekrashevych_palindromes} has constructed the first examples of simple groups of intermediate growth, which are of course not in $\MF$.
Nevertheless, these examples are far away from being branch and,   since Pervova only deals with torsion groups, it is also natural to wonder what happens for the non-torsion examples generalizing Grigorchuk and Gupta--Sidki groups.
%

We provide an answer to these questions by examining a family of self-similar groups (see Definition \ref{def:self-similar}) introduced by {\v{S}}uni{\'c} in \cite{Sunic},
as ``siblings of the first Grigorchuk group''. 
They do not seem to have a name in the literature, so we call them \emph{\v{S}uni\'{c} groups}. 
They are all branch groups (\cite[Theorem 1]{Sunic}) and we show in Theorem \ref{thm:justinfinite} that they are just infinite (every proper quotient is finite). 
Each group is defined by a prime $p$ (the degree of the rooted regular tree on which it acts faithfully) and a monic polynomial with coefficients in $\FF_p$.

It was later discovered (see \cite[Theorem 5.6]{NekrashevychSunic}) that for $p=2$, these groups have a nice dynamical characterization:
they are all the iterated monodromy groups%
\footnote{This is not used in our proofs so we omit any further reference to iterated monodromy groups in what follows and 
point the interested reader to \cite{NekrashevychSunic} and references therein for an introduction to this rich research area between group theory and dynamics.}
of the tent map. 
In fact, they are the residually finite (as opposed to simple) fragmentations of the infinite dihedral group (see \cite[Section 3]{Nekrashevych_palindromes}). 

\v{S}uni\'{c} groups can be either torsion or non-torsion.
The torsion ones are all of intermediate growth by \cite{BartholdiSunic}. 
The growth of the non-torsion ones is not completely settled but those acting on the binary tree are of intermediate growth
(except for the degenerate case of the infinite dihedral group),
by arguments very similar to those used by Grigorchuk in \cite{Grigorchuk_intermediate}. 
The most well-known example of one of these non-torsion groups of intermediate growth acting on the binary tree
is the second self-similar group in the family introduced by Grigorchuk (\emph{ibid.}).
Erchsler gave bounds for its growth in \cite{Erschler}, which is why this group is sometimes called the Grigorchuk--Erschler group in the literature.
 The main result in this paper is the following theorem.

\begin{thm}\label{thm:maximal}
 Let $G$ be a non-torsion {\v{S}}uni{\'c}  group that is not isomorphic to the infinite dihedral group and acts on the binary tree.  
The set of maximal subgroups of $G$ is countably infinite. 
Of these, $2^{m+1}-1$ are of finite index (in fact, index 2) where $m$ is the degree of the polynomial defining $G$.
The rest, of infinite index, are conjugates of finitely generated subgroups parametrised by all odd primes.
\end{thm}

The finitely generated subgroups parametrised by all odd primes are described in Section \ref{sec:dense}. 
Moreover, we show in Proposition \ref{prop:HqConjugatetoG} that these subgroups are not only isomorphic but in fact conjugate to $G$ in the full automorphism group of the binary tree.
Thus we give a complete description of all maximal subgroups of these {\v{S}}uni{\'c}  groups. 

The requirement that there be only countably infinitely many maximal subgroups seems quite strong. 
Other examples of finitely generated groups that satisfy it are Tarski monsters and some 3-soluble groups%
\footnote{This fact was pointed out to us by Cornulier.}
constructed by Hall in \cite{Hall} and considered by Cornulier in \cite{Cornulier}.

The theorem is a combination of several results. 
The number of maximal subgroups of finite index is obtained in Proposition \ref{prop:max_fin_index}.
Theorem \ref{thm:proper} yields countably infinitely many finitely generated subgroups of infinite index in $G$  which are contained in different maximal subgroups of infinite index.
Theorem \ref{thm:maximal_proof} states that these finitely generated subgroups are actually maximal.
Finally, by Theorem \ref{thm:allmaxareconjtoH}, each maximal subgroup of infinite index in $G$ is a conjugate in $G$ of one of the subgroups in Theorem \ref{thm:proper}.

Our approach follows similar principles to those of previous results on maximal subgroups in that it uses the profinite topology. 
Recall that a subgroup $H\leq G$ is dense in the profinite topology if $HN=G$ for every finite index normal subgroup $N$ of $G$. 
A proper subgroup $H$ which is dense in the profinite topology must be of infinite index (otherwise, there exists some finite index normal subgroup $N$ of $G$ contained in $H$, and so $HN=H\neq G$). 
Since all the groups we consider are finitely generated, any proper dense subgroup is contained in a maximal one, which is a fortiori dense and therefore of infinite index. 

Therefore, in order to show that the groups to which the theorem above applies are not in $\MF$, it suffices to find proper subgroups which are dense in the profinite topology.
This is a priori quite difficult as we would need to know every finite quotient of the group. 
However, it turns out that it suffices to consider the obvious finite quotients, those by level stabilizers.
A group of rooted tree automorphisms is said to have the congruence subgroup property if every finite index subgroup contains some level stabilizer.
We prove that, except for the degenerate case of the infinite dihedral group, all {\v{S}}uni{\'c} groups (for all primes $p$) have this property.

\begin{thm}\label{thm:csp}
 Let $G$ be a {\v{S}}uni{\'c} group which is not the infinite dihedral group. 
 Then $G$ has the congruence subgroup property: every finite index subgroup of $G$ contains some level stabilizer $\St_G(n)$. 
\end{thm}

The congruence subgroup property implies that in order to show that $H\leq G$ is dense in the profinite topology it suffices to show that $H\St_G(n)=G$ for every $n\in\NN$.
That is, that $H$ has the same action as $G$ on every level of the tree. 
We then use Proposition \ref{PH}, which is a nice application of Bézout's Lemma due to P.-H. Leemann, 
to obtain (in Corollary \ref{cor:dense}) for each prime $q$ different from $p$ a finitely generated subgroup $H(q)$ of $G$ which is dense in the profinite topology.
It therefore suffices to show that each $H(q)$ is a proper subgroup. 
We do this for  non-torsion {\v{S}}uni{\'c} groups acting on the binary tree in Section \ref{sec:proper}, obtaining Theorem \ref{thm:proper}, by examining their action on the boundary of the tree.
Since each $H(q)$ is contained in a maximal subgroup, and the product of two such subgroups is the whole group, we obtain countably many distinct maximal subgroups of infinite index.

Let us remark that if $p$ is odd, then $H(q)$ is not necessarily a proper subgroup.
Indeed, the Fabrykowski--Gupta group \cite{FabrykowskiGupta1} is a non-torsion {\v{S}}uni{\'c} group without proper dense subgroups in the profinite topology.
(The proof of this fact is part of a work in progress and will be written up elsewhere). 
It thus provides an example of a non-torsion group of intermediate growth in $\MF$.
Of course, in the case of torsion {\v{S}}uni{\'c} groups, the subgroups $H(q)$ are never proper.
An answer to the following question would be very interesting.  
\begin{qu}
Are all torsion \v{S}uni\'{c} groups in $\MF$?
\end{qu} 

As a consequence of Theorem \ref{thm:csp}, the non-torsion {\v{S}}uni{\'c} groups (the criterion for torsion is given in \cite{Sunic}, quoted here as Proposition \ref{prop:torsion_criterion}),
are examples of finitely generated residually finite non-torsion groups all of whose finite quotients are $p$-groups
(i.e. the profinite and pro-$p$ completions coincide). 
Such examples were also obtained in \cite{AlcoberGarridoUria} by showing that all Grigorchuk--Gupta--Sidki groups (except a degenerate example) have the congruence subgroup property.

The proof of Theorem \ref{thm:csp} moreover yields that all {\v{S}}uni{\'c} groups which are not isomorphic to the infinite dihedral group are just infinite. 
Since the dihedral group is also just infinite we obtain:
\begin{thm}\label{thm:justinfinite}
 All {\v{S}}uni{\'c} groups are just infinite.
\end{thm}

\begin{cor}
 Let $G$ be as in Theorem \ref{thm:maximal}, then $G$ is a primitive group; that is, the primitive action of $G$ on the cosets of any of its maximal subgroups of infinite index is faithful. 
\end{cor}

The theorem also has consequences for the Frattini subgroup. 
Recall that the Frattini subgroup $\Phi(\Gamma)$ of a group $\Gamma$ is the intersection of all maximal subgroups of $\Gamma$
(and defined to be $\Gamma$ if $\Gamma$ has no maximal subgroups). 
It is a characteristic subgroup and coincides with the set of non-generators of $\Gamma$; that is, those elements that are redundant in any generating set.
Thus the rank of $\Gamma$ is equal to the rank of $\Gamma/\Phi(\Gamma)$. 
For instance, if $\Gamma$ is nilpotent then $\Phi(\Gamma)$ contains the commutator subgroup  and so the rank of $\Gamma$ is the rank of its abelianization. 
The same holds for the  known examples of branch groups in $\MF$: 
since they are $p$-groups or have the congruence subgroup property (or both), all maximal subgroups have index $p$ and therefore contain the commutator subgroup. 
(Indeed, all known examples of branch groups in $\MF$ are in $\mathcal{MN}$, the class of groups all of whose maximal subgroups are normal, see \cite{Aglaia}.)
On the other hand, Theorems \ref{thm:maximal} and \ref{thm:justinfinite} immediately yield the following.
\begin{cor}
 The Frattini subgroup of any group as in Theorem \ref{thm:maximal} is trivial. 
\end{cor}
In this respect, the groups in question are similar to many of the ``geometric'' groups studied in \cite{GelanderGlasner}, 
which were shown there to be primitive and 
whose Frattini subgroups have also been shown to be small 
(see \cite{GelanderGlasner} and references therein). 

It was shown in \cite{BouRabee} that, although some branch groups are in $\MF$, all regular branch groups (see Definition \ref{def:branch_def}) contain uncountably many weakly maximal subgroups
(subgroups which are maximal with respect to being of infinite index). 
Since {\v{S}}uni{\'c} groups are regular branch, this was the motivation for proving Theorem \ref{thm:allmaxareconjtoH},
which implies that the non-torsion ones acting on the binary tree have exactly countably infinitely many maximal subgroups of infinite index. 
Thus there are many more weakly maximal subgroups than maximal ones.

The rest of the paper is organized as follows. 
Section \ref{sec:prelims} contains definitions and basic results on branch groups, self-similar groups and {\v{S}}uni{\'c} groups.
It also contains a description of all rigid stabilizers of \v{S}uni\'c groups acting on the binary tree.  
Section \ref{sec:csp} is dedicated to the proofs of Theorems \ref{thm:csp} and \ref{thm:justinfinite}. 
In Section \ref{sec:dense} we prove Proposition \ref{PH}, define the dense subgroups $H(q)$ and show that they are conjugate to $G$ in $\Aut T$.
We restrict to non-torsion \v{S}uni\'c groups acting on the binary tree in Section \ref{sec:proper}
where we show in Theorem \ref{thm:proper} that each $H(q)$ for $q\geq 3$  odd is a proper subgroup of $G$.
Section \ref{sec:Pervova} concerns more general self-replicating branch groups and contains
 a generalization of a result of Pervova, Proposition \ref{prop:PervovaProperProjection}.
This states that for suitable branch groups acting on the $p$-regular rooted tree, projections of proper dense subgroups are dense and proper. 
This is required in Section \ref{sec:maximal} where we return to the setting of Section \ref{sec:proper} and 
prove Theorem \ref{thm:maximal_proof}: each $H(q)$ is maximal.
Finally, Section \ref{sec:countably} contains the proof of Theorem \ref{thm:allmaxareconjtoH}:
for a non-torsion {\v{S}}uni{\'c} group acting on the binary tree, each maximal subgroup of infinite index is conjugate to some $H(q)$.

\medskip
\textbf{Acknowledgments} The main ideas in this paper developed during the course of a reading group on maximal subgroups of branch groups at the University of Geneva, 
suggested by Tatiana Nagnibeda, while the second author was a postdoc there.
We are grateful to the other participants of the reading group, Paul-Henry Leemann and Aitor P\'{e}rez for valuable ideas and conversations. 
We are also grateful to Tatiana Nagnibeda, Pierre de la Harpe and Laurent Bartholdi for useful comments and discussions and to the two anonymous referees for the many suggestions for improvements in the exposition. 
Both authors were supported by the Swiss National Science Foundation. 
The first author is also supported by the Natural Sciences and Engineering Research Council of Canada 
and the second was supported by the Alexander von Humboldt Foundation. 


\section{Preliminaries}\label{sec:prelims}
 
 \subsection{Regular rooted trees and their automorphism groups}
 For a fixed prime $p$, let $T$ be the regular rooted tree whose vertices are identified with elements of the monoid $X^*$ consisting of finite words over 
 the alphabet $X=\{\boldsymbol{0},\dots,\boldsymbol{p-1}\}$,
 and where two vertices $u$ and $v$ are joined by an edge if $v=ux$ or $u=vx$ for some $x\in X$.
 Notice that the elements $\boldsymbol{0}, \dots, \boldsymbol{p-1}$ of $X$ are written in bold font to avoid confusion. 
For every integer $n\ge 0$, the set $X^n$ of all vertices of length $n$ is called  \emph{level $n$} of $T$. 
We denote by $\Aut T$ the group of automorphisms of $T$.\\

\noindent\textbf{Convention.}
 We will use the convention that $\Aut T$ acts on $T$ on the left.
  Thus, for $g,h\in\Aut T$ and $v$ a vertex of $T$, we have
\[gh(v) = g(h(v)).\]
However, note that, in order to be consistent with most of the literature on the subject, we will use the group-theoretic conventions
\[g^h:=h^{-1}gh \qquad \text{ and } \qquad [g,h]:=g^{-1}h^{-1}gh.\]

We will denote by $\St(v)$ the stabilizer of a vertex $v$ and by $\St(n)$ the stabilizer of all vertices in $X^n$. 
If $G$ is a subgroup of $\Aut T$, 
we define $\St_G(n)$ and $\St_G(v)$ as the intersection with $G$ of the corresponding stabilizer in $\Aut T$.

Each vertex $v$ of $T$ is the root of a tree $T_v$ which is isomorphic to $T$ and which we henceforth identify with $T$.
So, for any $g\in \Aut T$, there is a unique $g_v\in \Aut T$ such that
\[g(vw) = g(v)g_v(w)\]
for all $w\in T$.
 We will call the automorphism $g_v$ the \emph{projection} of $g$ at $v$ (in the literature, $g_v$ is also called the \emph{section} or the \emph{state} of $g$ at $v$.)

For $n\in \NN$, let us denote by $\Aut X^n$ the image of $\Aut T$ in $\Sym(X^n)$. It is easy to see that $\Aut X^n= \Sym(X) \wr \dots \wr \Sym(X)$
is the $n$-fold iterated permutational wreath product of $\Sym(X)$. 

For all $n\in \NN$, we have an isomorphism
\begin{align*}
\psi_n\colon \Aut T &\rightarrow \Aut X^n \wr \Aut T\\
g&\mapsto \tau (g_v)_{v\in X^n}.
\end{align*}
In particular, $\psi_0$ is just the identity map on $\Aut T$ and $\psi_1$, which we will henceforth denote by $\psi$ for brevity, is the map
\begin{align*}
\psi\colon \Aut T &\rightarrow \Sym(X) \wr \Aut T \\
g &\mapsto \tau (g_{\boldsymbol{0}}, \dots, g_{\boldsymbol{p-1}}).
\end{align*}
When the permutation $\tau$ is trivial, we will omit it from the notation. Also, when convenient, we will sometimes omit the $\psi$ and write simply
$g = \tau (g_{\boldsymbol{0}}, \dots, g_{\boldsymbol{p-1}}).$
We say that the automorphism $g$ is \emph{rooted} if $g=\tau(1,\dots,1)$; that is, $g$ only permutes (rigidly) the subtrees rooted at level 1.

According to our convention, for any $g=\tau(g_{\boldsymbol{0}},\dots, g_{\boldsymbol{p-1}})$, $h=\sigma(h_{\boldsymbol{0}},\dots, h_{\boldsymbol{p-1}}) \in \Aut T$ we have
\begin{align*}
g^{-1} &=\tau^{-1}(g^{-1}_{\tau^{-1}(\boldsymbol{0})}, \dots, g^{-1}_{\tau^{-1}(\boldsymbol{p-1})})\\
\text{ and }\qquad gh &=\tau\sigma\sigma^{-1}(g_{\boldsymbol{0}},\dots, g_{\boldsymbol{p-1}})\sigma (h_{\boldsymbol{0}},\dots, h_{\boldsymbol{p-1}})\\
 &= \tau\sigma(g_{\sigma(\boldsymbol{0})}h_{\boldsymbol{0}}, \dots, g_{\sigma(\boldsymbol{p-1})}h_{\boldsymbol{p-1}}) 
\end{align*}

For each vertex $v$ of $T$, we can define a homomorphism
\begin{align*}
\varphi_v\colon \St(v) &\rightarrow \Aut T \\
g &\mapsto g_v.
\end{align*}
Notice that the domain of $\varphi_v$ is $\St(v)$, not $\Aut T$. The reason for this is that while the map $g\mapsto g_v$ is well-defined for any $g\in \Aut T$, it is not an endomorphism of $\Aut T$.

Given $G\leq \Aut T$, define the vertex \emph{projection} of $G$ at $v$ by $G_v:=\varphi_v(\St_G(v))\leq \Aut T$. Once again, it is necessary to restrict ourselves to $\St_G(v)$ in order to obtain a subgroup of $\Aut T$. As a consequence, however, note that there might exist some $g\in G$ such that $g_v\notin G_v$.

Notice that in general, for a subgroup $G\leq \Aut T$, an element $g\in G$ and a vertex $v$ of $T$, there is no reason for the vertex projection $g_v$ to also belong to $G$.

\begin{defn}\label{def:self-similar}
A subgroup $G\leq \Aut T$ is \emph{self-similar} if $g_v\in G$ for every vertex $v$ of $T$ and every $g\in G$.
It is \emph{self-replicating} if $G_v=G$ for every vertex $v\in T$.
If $G$ is a self-similar finitely generated group and $l$ is the length function given by a finite generating set, we say that $G$ is \emph{contracting} if there exist $M, n, l_0 \in \NN$ such that
$l(g_v)\leq \frac{l(g)}{2}+M$ 
for all $v\in X^n$ and $g\in G$ such that $l(g)>l_0$.
Whether or not $G$ is contracting does not depend on the choice of the finite generating set (see \cite[Section 2.11]{Nekrashevych:self-similar}). 
\end{defn}
For an arbitrary subgroup $G\leq\Aut T$, the image $\psi_n(\St_G(n))\leq \Aut T\times \dots \times \Aut T$ needs not be a direct product.
\begin{defn}\label{def:branch_def}
	Define $\rist_G(n)$, the rigid stabilizer in $G$ of level $n$, as the largest subgroup of $\St_G(n)$ which maps onto a direct product under $\psi_n$.
	We have
	$$\rist_G(n) = \prod_{v\in X^n} \, \rist_G(v)$$
	where $\rist_G(v)$ is the subgroup of all $g\in\St_G(n)$ such that $\psi_n(g)$ has all coordinates equal to $1$ except, possibly, at position $v$.
	If $G$ acts transitively on all levels of $T$, we say that $G$ is a \emph{branch group} if $|G:\rist_G(n)|<\infty$ for all $n$.
	Branch groups can be more generally defined when the rooted tree $T$ is not regular but level-homogeneous (see \cite[Section 5]{Grigorchuk_branch}).
	
	We say that a self-similar group $G\leq \Aut T$ is a \emph{regular branch group} (over a subgroup $K$) if it acts transitively on all levels of $T$
	and there exists a subgroup $K\leq G$ of finite index such that 
	$\psi(K)\geq K\times \dots \times K. $
	In particular, this implies that a regular branch group is a branch group.
\end{defn}

\subsection{{\v{S}}uni{\'c} groups}
Self-similar subgroups of $\Aut T$ have for many years now provided examples of groups with striking properties.
One of the most well-known examples is the first Grigorchuk group \cite{Grigorchuk_burnside},
whose properties have been generalized in many directions, 
yielding the notions of branch groups, groups generated by bounded automata, groups with endomorphic presentation, etc.
Moreover, it is the first example of a group of intermediate word growth. 
In fact, Grigorchuk produced in \cite{Grigorchuk_intermediate} an uncountable family of groups of intermediate growth, each defined by an infinite sequence $\omega$ on three symbols. 
Up to isomorphism, only two of these groups are self-similar, the first Grigorchuk group, which is torsion, and the group studied by Erschler in \cite{Erschler}, which is not torsion.
In \cite{Sunic}, {\v{S}}uni{\'c} introduced, for each prime $p$, self-similar groups which are close generalizations of the first Grigorchuk group 
(in his words, ``siblings of the first Grigorchuk group, not just some distant relatives''). 
We call each of these examples a \emph{{\v{S}}uni{\'c} group}, and recall their definition and basic properties, on which we will rely heavily,  below. This is essentially the main results of Sections 3 and 5 of \cite{Sunic}.
The family of \v{S}uni{\'c} groups for $p=2$ includes the two self-similar groups in Grigorchuk's uncountable family mentioned above (see Examples \ref{ex:grig} and \ref{ex:erschler}).

Let $p$ be a prime, $m\geq 1$ be an integer, and $A$ and $B$ be, respectively, the abelian groups $\ZZ /p \ZZ$ and $(\ZZ /p\ZZ)^m$  with multiplicative notation. 
We may also at times think of $A$ as the field $\FF_p$ of $p$ elements and of $B$ as an $m$-dimensional vector space over $A$.
Let $\rho:B\rightarrow B$ be an automorphism of $B$ 
 and $\omega: B\rightarrow A$ a surjective homomorphism.
We now define an action of $A$ and $B$ on the $p$-regular rooted tree. 
Let a generator $a$ of $A$ act as the rooted automorphism $\sigma(1,\dots,1)$ 
where $\sigma=(\boldsymbol{0}\; \boldsymbol{1}\; \ldots\; \boldsymbol{p-1})$;
that is, 
$a(xv)=\sigma(x)v$ for all $x\in X$ and $v\in X^*$.
The action of $b\in B$ on $T$ is recursively defined by $\psi(b)=(\omega(b),1,\dots,1,\rho(b))$.

\begin{defn}
Define the group $G_{\omega, \rho}\leq \Aut T$ by $G_{\omega,\rho}=\langle A\cup B\rangle$ with the actions of $A$ and $B$ described above.
\end{defn}
\begin{rem}\label{rem:Contraction}
It is easy to see from the definition that each $G_{\omega, \rho}$ is self-similar.
A straightforward calculation shows that it is also contracting, satisfying $l(g_v)\leq (l(g)+1)/2$ for every element $g$ and $v\in X^1$, where $l$ is the length function induced by the generating set $A\cup B$.
\end{rem}

\begin{prop}[Proposition 2 of \cite{Sunic}]\label{prop:SunicProp2}
 The following are equivalent:
 \begin{enumerate}[label=(\roman*)]
  \item the action of $B$ on $T$ is faithful;
  \item no non-trivial orbit of $\rho$ is contained in $\ker(\omega)$;
  \item no non-trivial $\rho$-invariant subspace of $B$ is contained in $\ker(\omega)$;
  \item $B$ is $\rho$-cyclic, the minimal polynomial $f$ of $\rho$ is
  $f(x)=x^m+a_{m-1x^{m-1}}+\dots+a_1x+a_0$,
  and there is a basis of $B$ with respect to which the matrices of $\rho$ and $\omega$ are given by
  \begin{equation}\label{eqn:rho}
  M_{\rho}=\begin{pmatrix} 
  0 & 0 &\dots &0  &-a_0\\
  1 & 0 &\dots &0  &-a_1\\
  0 & 1 &\dots &0 &-a_2\\
  \vdots &\vdots &\ddots &\vdots &\vdots\\
  0 & 0 &\dots &1 &-a_{m-1}    
   \end{pmatrix} 
   \quad
  M_{\omega}=\begin{pmatrix}
                 0 &0 &\dots &0 &1
                \end{pmatrix}.
  \end{equation}
 \end{enumerate}
\end{prop}

Since we want to consider subgroups of $\Aut T$,  we only consider groups for which the above faithfulness condition holds. 
\begin{defn}[{\v{S}}uni{\'c} group]\label{def:SunicGroups}
Let $p$ be a prime and $f(x)=x^m+a_{m-1}x^{m-1}+\dots+a_1x+a_0$ an invertible polynomial over $\FF_p$ (i.e., $a_0\neq 0$). 
 The group $G_{p,f}$ is the group $G_{\omega,\rho}$ where $\omega$ and $\rho$ are given by \eqref{eqn:rho}.
 
 For each $i\in \ZZ$, put $B_i:=\rho^i(\ker(\omega))$. 
 Then the above definition implies that there is a sequence of elements $b_0,b_1,\dots,b_{m-2}\in B_0$ and an element $b_{m-1}\in B_1\setminus B_0$ such
that
\begin{align*}
 b_0 &= (1,\dots,1,b_1), \quad  b_1 = (1,\dots,1,b_2), \quad	  \dots, \quad
 b_{m-2} = (1,\dots,1,b_{m-1}),\\
 b_{m-1}&=(a,\dots,1,\rho(b_{m-1}))
\end{align*}
where $\rho(b_{m-1})=b_0^{-a_0}b_1^{-a_1}\cdots b_{m-1}^{-a_{m-1}}$.

Note that $B=\langle b_0,\dots, b_{m-1}\rangle$ while 
 $B_0 = \ker(\omega)=\langle b_0,\dots,b_{m-2}\rangle, \, B_1=\rho(B_0)=\langle b_1,\dots,b_{m-1}\rangle$.
\end{defn}
 
 \begin{example}\label{ex:dihedral}
 The group $G_{2,x+1}$ is the infinite dihedral group, generated by $a$ and $b=(a,b)$.
 \end{example}
 \begin{example}\label{ex:grig}
 The group $G_{2,x^2+x+1}$ is the first Grigorchuk group. 
 \end{example}
 \begin{example}\label{ex:erschler}
 The group $G_{2,x^2+1}$ is the other self-similar group in the uncountable family defined in \cite{Grigorchuk_intermediate}. 
 It was studied in \cite{Erschler} and it is not torsion.
  The standard matrices of $\rho$ and $\omega$ are 
  $$M_{\rho}=\begin{pmatrix}
     0 & 1\\
     1 & 0
    \end{pmatrix} 
    \quad
    M_{\omega}=(0,1)$$
    giving generators $a, b_0=(1,b_1),b_1=(a,b_0)$.
    Notice that $\langle a, b_0b_1=(a,b_0b_1)\rangle$ is isomorphic to the infinite dihedral group. 
 \end{example}
 \begin{example}
 The group $G_{3,x-1}$ is the Fabrykowski--Gupta group introduced in \cite{FabrykowskiGupta1}. It is generated by $a$ and $b=(a,1,b)$. 
 It was shown to be of intermediate growth in \cite{FabrykowskiGupta2} and \cite{BartholdiPochon}. 
 \end{example}

\begin{notation}
 Denote by $\mathcal{G}_{p,m}$ the family of groups $G_{p,f}$ where $f$ has degree $m$ and 
 by $\mathcal{G}$ the family of all {\v{S}}uni{\'c} groups $G_{p,f}$ for all primes $p$.
\end{notation}

Let us collect some useful results about {\v{S}}uni{\'c} groups.
\begin{prop}[Proposition 3 of \cite{Sunic}]\label{prop:SunicProp3}
 Let $f$ be a monic polynomial, invertible over $\FF_p$, which factors as $f=f_1f_2$,
for some non-constant monic polynomials $f_1, f_2$. 
Then $G_{p,f_i}\leq G_{p,f}$ for $i=1,2$.
\end{prop}

\begin{prop}[Proposition 9 of \cite{Sunic}]\label{prop:torsion_criterion}
 Let $G$ be a group in $\mathcal{G}_{p,m}$ with $m\geq2$. The following are equivalent:
 \begin{enumerate}[label=(\roman*)]
  \item $G$ is a $p$-group;
  \item there exists $r$ such that $B_0\cup B_1\cup \dots \cup B_{r-1}=B$;
  \item every non-trivial $\rho$-orbit intersects $B_0=\ker(\omega)$. 
 \end{enumerate}
\end{prop}

\begin{cor}\label{cor:contain_dihedral}
Let $f$ be a monic polynomial invertible over $\FF_2$ and let $G=G_{2,f}$. The following are equivalent :
\begin{enumerate}[label=(\roman*)]
\item \label{item:dihedral1}$G$ contains an element of infinite order;
\item \label{item:dihedral2}there exists $b\in B\setminus \bigcup_{n\in\NN}B_n $ such that $b=(a,b)$;
\item \label{item:dihedral3}$f$ is divisible by $x+1$;
\item \label{item:dihedral4}$G$ contains a copy of the dihedral group.
\end{enumerate}
\end{cor}
\begin{proof}
\ref{item:dihedral1} implies \ref{item:dihedral2}: By Proposition \ref{prop:torsion_criterion}, since $G$ contains an element of infinite order,
there exists an element $b\in B$ such that the $\rho$-orbit of $b$ does not intersect $B_0 = \ker(\omega)$. 
Hence, $\omega(\rho^k(b\rho(b))) = \omega(\rho^k(b))\omega(\rho^{k+1}(b)) = a^2 = 1$ for all $k\in \NN$. 
By (2) of Proposition \ref{prop:SunicProp2}, we get that $\rho(b)=b$, so $b=(a,b)$.

\ref{item:dihedral2} implies \ref{item:dihedral3}: Since $b=(a,b)$, we have $\rho(b)=b$, so $b$ is an eigenvector of $\rho$ with eigenvalue $1$. 
It follows that $f$, the minimal polynomial of $\rho$ 
 is divisible by $x+1$.

\ref{item:dihedral3} implies \ref{item:dihedral4} is a direct consequence of Proposition \ref{prop:SunicProp3} and Example \ref{ex:dihedral} while \ref{item:dihedral4} trivially implies \ref{item:dihedral1}. 
\end{proof} 
 
 The following is equivalent to Proposition 4 of \cite{Sunic}, but we give a different proof here. 
 \begin{prop}\label{prop:AbelianisationGIsAB}
  Let $G\in \mathcal{G}$. 
  Put $\Gamma:=A\ast B$  and denote by $\pi\colon\Gamma \rightarrow G$ the canonical map, with kernel $N$.
  Then $N\leq \Gamma'$, the commutator subgroup of $\Gamma$, and $A\times B\cong \Gamma/\Gamma'\cong G/G'$ where the isomorphisms are canonical.
 \end{prop}
\begin{proof}
	First note that $\pi^{-1}(G')=\Gamma'N$ and that 
	$$G/G'\cong(\Gamma/N)/(\Gamma'N/N)\cong \Gamma/(\Gamma'N)\twoheadleftarrow \Gamma/\Gamma'\cong A\times B.$$
	
	Consider the subgroup $S$ of index $p$ in $\Gamma$ generated by $\{a^{-i}xa^i\mid  x\in B, i\in\FF_p \}$. 
	Since $\pi(a)\notin \St_G(1)$, the kernel $N$ is contained in $S$ and $S$ is the preimage of $\St_G(1)$ in $\Gamma$.
	Let $\Psi\colon S \to \Gamma\times\cdots\times \Gamma$ ($p$ factors) be the homomorphism defined by
	\begin{align*}
	\Psi(x)& = (\omega(x),1,\dots,1,\rho(x)) \\
	\Psi(a^{-1}xa) &= (1,\dots,1, \rho(x), \omega(x))\\
	  &\cdots \\
	\Psi(axa^{-1}) &= (\rho(x), \omega(x), 1,\dots,1)
	\end{align*}
	for all $x\in B$.
	Defining $\pi_{G\times\cdots\times G} := \pi\times\cdots\times \pi$, we have 
	$\pi_{G\times\cdots\times G}\circ \Psi = \psi\circ \pi$,
	as the images of the generators of $S$ by each of the maps coincide.

	To show that $N\leq \Gamma'$, suppose that $\gamma\in N\setminus \Gamma'$, then $\gamma=a^i\beta z$ with $i\in\FF_p$, $\beta\in B$, (not both trivial) and $z\in \Gamma' $. 
	Since $\gamma\in N$ and $\beta z\in S$, we must have $i=0$. 
	Now, $$\Psi(\gamma)=(\omega(\beta)z_0, z_1,\dots, \rho(\beta)z_{p-1}) \text{ where  } \Psi(z)=(z_0,\dots,z_{p-1}).$$
	Since $z\in \Gamma'$, it is easily seen (by considering $[a,x]$ for $x\in B$) that $z_0z_1\cdots z_{p-1}\in \Gamma'$.
	Thus the product of all entries in $\Psi(\gamma)$ is congruent to $\omega(\beta)\rho(\beta)$ modulo $\Gamma'$. 
	But, since $\pi_{G\times\cdots\times G}\circ \Psi = \psi\circ \pi$, this product must be in $N$ and so $\omega(\beta)=1$. 
	Repeating the above argument, we obtain that $\omega(\rho^n(\beta))=1$ for all $n\in \NN$, which implies that $\beta=1$ and so $\gamma=z\in\Gamma'$, as required. 	
\end{proof}
 
 \begin{prop}[Proposition 5 of \cite{Sunic}]
  Let $G\in \mathcal{G}$. Then:
  \begin{enumerate}[label=(\roman*)]
   \item The map $\psi$ induces a subdirect embedding of $\St_G(1)$ in $G\times\dots\times G$.
   \item $G$ acts transitively on all levels of $T$.
  \end{enumerate}
 \end{prop}

%
 %
 %
 %
\begin{prop}[Lemmas 1 and 6 of \cite{Sunic}]\label{prop:regular_branch}\hfill
 \begin{enumerate}[label=(\roman*)]
  \item Let $G\in\mathcal{G}_{p,m}$ where $p\geq 3$. Then $G$ is regular branch over its commutator subgroup $G'$.
  \item Let $G\in \mathcal{G}_{2,m}$ where $m\geq 2$. Then $G$ is regular branch over the subgroup
  $$K:=\langle [a, b]\mid b\in B_1\rangle^G.$$
 \end{enumerate}
\end{prop}
%
%
%
%
%
\begin{lem}[See Lemmas 3, 5 and 7 of \cite{Sunic}]\label{lemma7}
Let $G\in \mathcal{G}_{2,m}$ with $m\geq 2$ and denote by $\overline{B_1}$ the normal closure of $B_1$ in $G$.
\begin{enumerate}[label=(\roman*)]
 \item For any element $d\in B_0\setminus B_1$, we have
	\begin{equation}\label{eqn:lemma3}
	G=\langle a, d\rangle \ltimes \overline{B_1}=\langle a, d \rangle \ltimes (B_1\ltimes K)
	\end{equation}
	
 \item There is an element $c\in B_{-1}\setminus B_0$ and $d\in B_0\setminus B_1$ such that 
 $c=(a,d)$ and 
\begin{equation}\label{eqn:lemma7}
	\psi(\St_G(1))=\hat{C} \ltimes (\overline{B_1}\times \overline{B_1})
\end{equation}
 where $\hat{C}=\langle (a,d),(d,a)\rangle$ is a diagonal subgroup of $\langle a, d\rangle \times \langle a, d\rangle$.
 
  In particular, if $g\in \St_G(1)$ is such that $\psi(g) = (h, 1)$ or $\psi(g) = (1,h)$ with $h\in \langle a,d \rangle \leq G$ then, $g=1$.

\end{enumerate}
\end{lem}

%
%
The next results will be useful in Section \ref{sec:maximal} when we show that the subgroups defined in Section \ref{sec:dense} are indeed maximal and of infinite index.
\begin{prop}\label{prop:liftIsHomomorphism}
	Let $G\in\mathcal{G}_{2,m}$ with $m\geq 2$ and let $c, d\in G$ be as in Lemma \ref{lemma7}.
	There exists a unique homomorphism $\phi\colon G \to \St_G(1)$ such that
	\begin{align*}
	\phi(a) &= aca\\
	\phi(x) &= \rho^{-1}(x)
	\end{align*}
	for all $x\in B$.
\end{prop}
\begin{proof}
	If such a homomorphism exists, then it is clearly unique. 
	Thus, it suffices to show that the above yields a well-defined homomorphism.
Put $\Gamma:=A\ast B$ and consider the homomorphisms $\pi, \Psi, \pi_{G\times G}$ defined in the proof of Proposition \ref{prop:AbelianisationGIsAB}. 

	Let $\Phi\colon A\ast B \to S$ be the homomorphism defined by
	\begin{align*}
	\Phi(a) &= aca\\
	\Phi(x) &= \rho^{-1}(x)
	\end{align*}
	for all $x\in B$.
	Defining $\phi:=\pi_G\circ \Phi\circ \pi_G^{-1}$, we obtain the following diagram, where the bottom square commutes.
	
	\[
	\begin{tikzcd}
	A\ast B \arrow[twoheadrightarrow]{r}{\pi_G} \arrow{d}{\Phi} & G \arrow{d}{\phi}\\
	S \arrow[twoheadrightarrow]{r}{\pi_G} \arrow{d}{\Psi} & \St_G(1) \arrow{d}{\psi} \\
	(A\ast B)\times (A\ast B)\arrow[twoheadrightarrow]{r}{\pi_{G\times G}} & G\times G
	\end{tikzcd}
	\]

	To show that $\phi$  is a well-defined homomorphism, it suffices to show that $\Phi(N)\leq N$. 	
	A direct computation shows that every element $w$ of $A$ or $B$ in $\Gamma=A\ast B$ satisfies $\Psi(\Phi(w))=(w',w)$ with $w'\in \langle a, d\rangle\leq\Gamma$. 
	Therefore every element $w$ of $\Gamma$ satisfies the relation too. 
	So if $w\in N$ then there exists $w'\in\langle a, d\rangle\leq \Gamma$ such that $\Psi(\Phi(w)) = (w',w)$. Hence,
	\begin{equation*}
	\pi_{G\times G}(\Psi(\Phi(w))) = (\pi(w'),1)
	\end{equation*}
	with $\pi(w') \in \langle a,d \rangle \leq G$. 
	Therefore, since $\pi_{G\times G}\circ \Psi=\psi\circ\pi$, we get
	\begin{equation*}
	\psi(\pi(\Phi(w))) = (\pi(w'),1)
	\end{equation*}
	with $\pi(w') \in \langle a,d \rangle \leq G$. 
	It follows from Lemma \ref{lemma7} that $\pi(\Phi(w)) = 1$.
\end{proof}

\begin{rem}
	With the notation of the previous lemma, for all $g\in G$, we have $\psi(\phi(g)) = (g',g)$ for some $g'\in \langle a,d \rangle \leq G$.
	Hence, $\phi$ is a right inverse of the projection $\varphi_{\boldsymbol{1}}$ on the second coordinate.
\end{rem}

\begin{prop}\label{prop:LiftIsUnique}
	Keep the notation of the previous two results and let $g\in G$.
	 If there exists $h\in \St_G(1)$ such that $\psi(h) = (1,g)$, then $h = \phi(g)$.
\end{prop}
\begin{proof}
	We have $\psi(\phi(g)) = (x,g)$, with $x\in \langle a,d \rangle$. 
	Therefore, $\psi(\phi(g)h^{-1}) = (x,1)$. 
	It follows from Lemma \ref{lemma7} that $\phi(g) = h$.
\end{proof}

\subsection{The difference with previous examples of branch groups outside $\MF$}

We mentioned in the Introduction that we are not providing the first examples of finitely generated branch groups which are not in $\MF$. 
The first examples were found by Bondarenko in \cite{Bondarenko}, using similar constructions to those in \cite{Neumann, Segal, Wilson}. 
In these examples, the existence of a maximal subgroup of infinite index is guaranteed by the fact that the subgroup of  finitary automorphisms is proper and dense in the given branch group.
A \emph{finitary} automorphism is one which admits non-trivial projections at only finitely many vertices.
Its \emph{depth} is the first level at which all its projections are trivial. 
By contrast, the subgroup of finitary automorphisms of a group in $\mathcal{G}_{2,m}$ is very small indeed.
\begin{prop}
 Let $G\in \mathcal{G}_{2,m}$. Then $A=\langle a \rangle $ is the subgroup of finitary automorphisms of $G$. 
\end{prop}
\begin{proof}
Let $g\in G$ be a finitary automorphism. It can be written as $g=a^k(g_{\boldsymbol{0}},g_{\boldsymbol{1}})$, where $k\in\{0,1\}$ and $g_{\boldsymbol{0}}, g_{\boldsymbol{1}}\in G$ are finitary automorphisms. If the depth of $g$ is $n\in \NN^*$, then the depths of $g_{\boldsymbol{0}}$ and $g_{\boldsymbol{1}}$ are at most $n-1$. In particular, if $g$ has depth $1$, then $g=a(1,1) = a \in A$.

We will now show that there can be no finitary automorphisms of depth $n$ for $n>1$. If $g=a^k(g_{\boldsymbol{0}},g_{\boldsymbol{1}})$ has depth $n$, then one of $g_{\boldsymbol{0}}$ or $g_{\boldsymbol{1}}$ must have depth exactly $n-1$. Therefore, by induction, it suffices to show that there are no finitary automorphisms of depth $2$.

For the sake of contradiction, let us assume that there exists $g=a^k(g_{\boldsymbol{0}},g_{\boldsymbol{1}})$ such that $g$ has depth exactly $2$. Since $a^kg$ is also a finitary automorphism of depth $2$, we can assume without loss of generality that $g\in \St_G(1)$. We know that one of $g_{\boldsymbol{0}}$ or $g_{\boldsymbol{1}}$ must be of depth $1$, and therefore be equal to $a$. Without loss of generality (as conjugation by $a$ does not change the depth of $g$), let us assume that $g_{\boldsymbol{0}}=a$. Since $g_{\boldsymbol{1}}$ is of depth at most $1$, either $g_{\boldsymbol{1}}=1$ or $g_{\boldsymbol{1}}=a$. 
In either case, for the $c$ and $d$ of Lemma \ref{lemma7}, we have $gc=(1,g_{\boldsymbol{1}}d)\in \St_G(1)$ with $g_{\boldsymbol{1}}d\in \langle a, d\rangle$. Therefore, by Lemma \ref{lemma7}, we have $g_{\boldsymbol{1}}d=1$, a contradiction.
We conclude that there exists no finitary automorphism of depth $2$, and thus that the subgroup of finitary automorphisms of $G$ is $A$.
%
%
%
%
\end{proof}

The fact that the finitary automorphisms form a dense subgroup  in the examples considered by Bondarenko is a consequence of the fact
that $\varphi_v(\rist_G(v))=G$ for every $v\in T$ where $G$ is one of these examples. 
Of course, this condition does not hold for {\v{S}}uni{\'c} groups (this would contradict the previous result),
but we can say precisely what the rigid vertex stabilizers are. 

\begin{prop}\label{rists}
 Let $G\in\mathcal{G}_{2,m}$ with $m\geq 2$. Denote by $\overline{B_1}$ the normal closure of $B_1$ in $G$ and recall the definition of $K$ from Proposition \ref{prop:regular_branch}. 
 Then 
 \begin{enumerate}[label=(\roman*)]
  \item\label{item:rist1} $\varphi_v(\rist_G(v))=\overline{B_1}$ for $v\in X$
  \item\label{item:ristn} $\varphi_v(\rist_G(v))=\varphi_v(\rist_K(v))=K$ for any $v\in X^n$ with $n\geq m$. 
 \end{enumerate} 
\end{prop}

\begin{proof}

 \ref*{item:rist1} If $b\in B_0$ then $\psi(b)=(1,\rho(b))$ and $\psi(b^a)=(\rho(b),1)$ so, since $\psi(\St_G(1))$ is subdirect in $G\times G$ we have
 $\varphi_v(\rist_G(v))\geq \varphi_v(\langle B_0\rangle^{\St(1)})=\overline{B_1}$. 
 On the other hand, by Lemma \ref{lemma7}, $G=\langle a,d \rangle \ltimes \overline{B_1}$ and  $\varphi_v(\rist_G(v))\cap \langle a,d\rangle = \{1\}$.
 The claim follows. 
 
 \ref*{item:ristn} The fact that $G$ is regular branch over $K$ implies that $\varphi_v(\rist_K(v))\geq K$ for every $v\in T$. 
 Put $K_0:=G $ and $K_n:=(B_1\cap \dots\cap B_n)\ltimes K$ for $n\geq 1$, which is indeed a semi-direct product by Lemma \ref{lemma7}.
 We claim that $\varphi_{\boldsymbol{1}^n}(\rist_G(\boldsymbol{1}^n))=\varphi_{\boldsymbol{1}}(\rist_{K_{n-1}}(\boldsymbol{1}))=K_n$ for every $n\geq 1$. 
 The case $n=1$ follows from the previous item because $K_1 = B_1\ltimes K \leq \overline{B_1}$ by definition of $K$, and it can be easily verified that $K_1$ is normal in $G$, which implies that $K_1=\overline{B_1}$.

 Let $n\geq 2$ and suppose the claim true for $n-1$. By inductive hypothesis,
 \begin{align*}
 \varphi_{\boldsymbol{1}^n}(\rist_G(\boldsymbol{1}^n))&=\varphi_{\boldsymbol{1}}(\rist_G(\boldsymbol{1})\cap \varphi_{\boldsymbol{1}^{n-1}}(\rist_G(\boldsymbol{1}^{n-1})))\\
    &=\varphi_{\boldsymbol{1}}(\rist_G(\boldsymbol{1})\cap K_{n-1})=\varphi_{\boldsymbol{1}}(\rist_{K_{n-1}}(\boldsymbol{1})).
 \end{align*}
  So we must show that $\varphi_{\boldsymbol{1}}(\rist_{K_{n-1}}(\boldsymbol{1}))=K_n$.
 We start by showing that 
 \[\varphi_{\boldsymbol{1}}(\rist_{K_{n-1}}(\boldsymbol{1}))\cap(B_1\cap\dots\cap B_{n-1})=B_1\cap\dots\cap B_n.\] 
 First note that if $b\in B_1\cap\dots\cap B_{n}$ then $\rho^{-1}(b)\in B_0\cap\dots\cap B_{n-1} \leq \rist_{K_{n-1}}(\boldsymbol{1})$ because $B_0\leq \rist_G(\boldsymbol{1})$ and $B_1\cap\dots\cap B_{n-1}\leq K_{n-1}$ by definition. 
 For the other inclusion, let $b\in (B_1\cap\dots\cap B_{n-1})\setminus B_n$;
 then 
\[\rho^{-1}(b)\in (B_0\cap\dots\cap B_{n-2})\setminus B_{n-1},\] 
so $\rho^{-1}(b)\notin K_{n-1}$ and so $b\notin \varphi_{\boldsymbol{1}}(\rist_{K_{n-1}}(\boldsymbol{1}))$.

 Now, $K_{n-1}\geq K$ so 
$$\varphi_{\boldsymbol{1}}(\rist_{K_{n-1}}(\boldsymbol{1}))\geq \varphi_{\boldsymbol{1}}(\rist_K(\boldsymbol{1}))\geq K$$
and thus $\varphi_{\boldsymbol{1}}(\rist_{K_{n-1}}(\boldsymbol{1}))\geq K_n$. 
 On the other hand, by inductive hypothesis, $K_{n-1}=\varphi_{\boldsymbol{1}}(\rist_{K_{n-2}}(\boldsymbol{1}))\geq \varphi_{\boldsymbol{1}}(\rist_{K_{n-1}}(\boldsymbol{1}))$, and so 
\begin{align*}
\varphi_{\boldsymbol{1}}(\rist_{K_{n-1}}(\boldsymbol{1})) &=\varphi_{\boldsymbol{1}}(\rist_{K_{n-1}}(\boldsymbol{1}))\cap K_{n-1}\\ 
	&=(\varphi_{\boldsymbol{1}}(\rist_{K_{n-1}}(\boldsymbol{1}))\cap B_1\cap \dots\cap B_{n-1}) \ltimes (\varphi_{\boldsymbol{1}}(\rist_{K_{n-1}}(\boldsymbol{1}))\cap K)\\
	&=(B_1\cap\dots\cap B_n)\ltimes K = K_n. 
 \end{align*}
 The claim then follows by induction.
%
%
%
%

 Since each $B_i$ has codimension 1 in $B$, the intersection $B_1\cap\dots\cap B_m$ is trivial. 
 Hence $\varphi_{\boldsymbol{1}^m}(\rist_G(\boldsymbol{1}^m))=(B_1\cap\dots\cap B_m)\ltimes K=K$ and the second item follows, using the fact that $G$ acts level-transitively and $K$ is normal in $G$.
\end{proof}

 
 \section{The congruence subgroup property}\label{sec:csp}
 
 This section contains the proofs of Theorems \ref{thm:csp} and \ref{thm:justinfinite}. 
 In fact, both theorems follow from the same fact. The main idea is explained in the following remark.

 \begin{rem}\label{prop:branch_csp}
 It is implicit in the proof of \cite[Theorem 4]{Grigorchuk_branch} that if $G \leq \Aut T$ is transitive on every level
 then every non-trivial normal subgroup of $G$ contains  $\rist_G(n)'$ for some $n$, where $'$ denotes the commutator subgroup.
 If $G$ contains a non-trivial subgroup $H$ such that $\psi(H)\geq H\times  \dots \times H$
 (that is, $G$ is weakly regular branch over $H$), then $H\leq \varphi_v(\rist_G(v))$ for each $v\in T$. 
 Thus if $H'$ contains some level stabilizer $\St_G(m)$ then for every non-trivial normal subgroup $N$ of $G$  there exists some $n$ such that
 \begin{equation*}
 \begin{split}
 N &\geq \rist_G(n)'\geq \psi_n^{-1}(H'\times \dots\times H')\\
 & \geq\psi_n^{-1}( \St_G(m)\times\dots\times\St_G(m))=\St_G(m+n).  
\end{split}
\end{equation*}

 In particular, $G$ is just infinite and has the congruence subgroup property. 
 \end{rem}
 

In view of this and Proposition \ref{prop:regular_branch}, Theorems \ref{thm:csp} and \ref{thm:justinfinite} will follow once we show that for $G\in \mathcal{G}_{p,m}$,
$G''$ (if $p\geq 3$) or $K'$ (if $p=2$ and $m\geq 2$) contains some level stabilizer.

Before delving into the proofs of these facts, let us remark the following consequence of Theorem \ref{thm:csp} for the maximal subgroups of \v{S}uni\'{c} groups. 
In our considerations we  exclude the infinite dihedral group (the case $p=2, m=1$), because it does not have the congruence subgroup property (it cannot, as it has finite quotients of odd order) and because its countably many maximal subgroups are all of finite index, and every prime is represented among these indices. 

\begin{prop}\label{prop:max_fin_index}
 Let $G$ be a \v{S}uni\'{c} group acting on the $p$-regular tree and with defining polynomial of degree $m$. If $G$ is not the infinite dihedral group, then $G$ has exactly $p^{m+1}-1$ maximal subgroups of finite index, each of index $p$. 
\end{prop}
\begin{proof}
%

Let $M$ be a maximal subgroup of finite index in $G$. 
By Theorem \ref{thm:csp}, $M$ contains some level stabilizer $\St_G(n)$. 
Since the quotient $G/\St_G(n)$ is a $p$-group, the image of $M$ is normal of index $p$ and therefore the same holds for $M$.
In particular, $G'\leq M$. 
Now, $G/G'$ is an elementary abelian $p$-group of rank $m+1$ and so the number of its maximal subgroups is $p^{m+1}-1$. 
The result follows. 
\end{proof}

\noindent\textbf{Notation.} Recall our convention that $g^h=h^{-1}gh$ and $[g,h]=g^{-1}h^{-1}gh$.

 \subsection{Odd prime case}
In this subsection, $G$ will denote a group in $\mathcal{G}_{p,m}$ where $p$ is an odd prime and $m\geq 1$. 
The proof that the second derived subgroup $G''$ of $G$ contains some level stabilizer is split into two lemmas.
%

\begin{lem}\label{G''>gamma}
Let $G\in \mathcal{G}_{p,m}$ where $p$ is an odd prime, then
 $$\psi(G'')\geq \gamma_3(G)\times \overset{p}{\dots} \times \gamma_3(G)$$
 where $\gamma_3(G)$ is the third term in the lower central series of $G$.
\end{lem}
\begin{proof}
 We first show that $\psi(G')\leq G\times\overset{p}{ \dots}\times G$ is subdirect.
 For any $b\in B$, we have 
 $\psi([a^{-1},b])=(\rho(b^{-1})\omega(b),\omega(b^{-1}),1,\dots,1,\rho(b))$. 
 Since $\rho$ is an automorphism of $B$, we can obtain all generators of $B$ in the last coordinate of $\psi(G')$. 
 Also, since there exists $b\in B$ such that $\omega(b)=a^{-1}$, we can obtain $a$ in the last coordinate of $\psi(G')$, by taking $[a^{-1},b]^{a^2}$.
 Thus $\psi(G')\leq G\times\overset{p}{ \dots}\times G$ maps onto $G$ in the last coordinate (and in all the others, by conjugating by suitable powers of $a$).
 
 The result follows from the above and the fact that $\psi(G')\geq G'\times\overset{p}{ \dots}\times G'$.
\end{proof}

\begin{lem}\label{gamma>G'}
Let $G\in \mathcal{G}_{p,m}$ where $p$ is an odd prime.
The second derived subgroup $G''$ of $G$ contains the stabilizer $\St_G(m+3)$.
%
%
\end{lem}
\begin{proof}
We start by showing that $\psi(\gamma_3(G))\geq G'\times\overset{p}{ \dots}\times G'$.
 Let $b\in B$ and $c\in B$ such that $\omega(c)=a$.
 Then, $x:=[[c,a],\rho^{-1}(b)]\in \gamma_3(G)$. Using classical commutator identities and the fact that elements of $B$ commute, we have
 \begin{align*}
 x&=[c^{-1}a^{-1}ca, \rho^{-1}(b)]\\
 &=[c^{-1}, \rho^{-1}(b)]^{a^{-1}ca}[a^{-1}ca,\rho^{-1}(b)] \\
 &=[a^{-1}ca, \rho^{-1}(b)].
 \end{align*} 
Therefore,
 \begin{align*}
    \psi(x) &=\psi([a^{-1}ca, \rho^{-1}(b)])\\
	    &=([1,\omega(\rho^{-1}(b)],1,\dots,1,[\rho(c),1],[\omega(c),b])\\
	    &=(1,\dots,1,1,[a,b]).
 \end{align*}
Since $G'=\langle [a,b]\mid b\in B\rangle^G$, and $\varphi_{\boldsymbol{p-1}}(\St_G(\boldsymbol{p-1}))=G$, we can conjugate $\gamma_3(G)$ by elements in $\St_G(\boldsymbol{p-1})$ to obtain that
$\gamma_3(G)\geq 1\times\dots \times 1\times G'$. 
Conjugating by suitable powers of $a$ we conclude that $\psi(\gamma_3(G))\geq G'\times\overset{p}{ \dots}\times G'$, as required.

It was shown in \cite[Lemma 9]{Sunic} that $G'$ contains $\St_G(m+1)$. 
We therefore have, by Lemma \ref{G''>gamma}, 
\begin{align*}
    \psi_2(G'') & \geq \psi(\gamma_3(G))\times \overset{p}{\dots}\times \psi(\gamma_3(G))\\
	& \geq G'\times \overset{p^2}{\dots}\times G' \\
	& \geq \St_G(m+1)\times \overset{p^2}{\dots} \times \St_G(m+1)=\psi_2(\St_G(m+3)).
 \end{align*}
The claim then follows as $\psi_2$ is injective. 
\end{proof}



 \subsection{Even prime case}
 As explained at the start of the section, we must exclude the infinite dihedral group from our considerations here.
%
%
 Thus in this subsection $G$ will denote a group in $\mathcal{G}_{2,m}$ with $m\geq 2$.
 We will show that $K'$ contains some level stabilizer in two steps. 
Recall that $K_v$ denotes $\varphi_v(\St_K(v))$ for a vertex $v\in T$ and that $v$ is thought of as a word in the alphabet $X=\{\boldsymbol{0},\boldsymbol{1}\}.$
%
%
%
 \begin{lem}\label{K>B0}
 Let $G\in \mathcal{G}_{2,m}$ with $m\geq 2$. 
 There exists $n\in\NN$ such that $K_v\geq \langle a, B_0\rangle$ where $v=\boldsymbol{1}^n.$
 \end{lem}
 \begin{proof}
 Recall from Definition \ref{def:SunicGroups} that $B=\langle b_0, \dots,b_{m-1}\rangle$ and $B_1=\langle b_1, \dots,b_{m-1}\rangle$.
 Since $K=\langle [a,b_1],\dots,[a,b_{m-1}]\rangle^G$  it suffices to examine the projections of $[a,b_i]$ for $i=1,\dots,m-1$.
 We have \[\varphi_{\boldsymbol{1}}([a,b_i])= \begin{cases}
                   b_{i+1} & \text{ if } i=1,\dots,m-2 \\
                   a\rho(b_{m-1}) & \text{ if } i=m-1.
                  \end{cases}\]
  Hence
   \[\varphi_{\boldsymbol{11}}([a,b_i])=\begin{cases}
                       b_{i+2} & \text{ if } i=1,\dots,m-3\\
                       \rho(b_{m-1}) & \text{ if } i= m-2
                      \end{cases}\]
and 
\[\varphi_{\boldsymbol{11}}([a,b_{m-1}]^2)=\omega(\rho(b_{m-1}))\rho^2(b_{m-1})\]

so $K_{\boldsymbol{11}}$ contains $b_3,\dots,b_{m-1},\rho(b_{m-1}),\omega(\rho(b_{m-1}))\rho^2(b_{m-1})$.

  \medskip
  \noindent
  \emph{Case $m>2$:} Put $x:=[a, b_{m-2}]^{[a,b_{m-1}]}$. A direct computation shows that 
  \[\varphi_{\boldsymbol{11}}(x^a)=a.\]
  Hence $K_{\boldsymbol{11}}$ contains $a, b_3,\dots,b_{m-1},\rho(b_{m-1}), \rho^2(b_{m-1})$ and therefore  $\langle a, \rho^2(B_0)\rangle$.
  As $\varphi_{\boldsymbol{1}}(\rho^2(B_0))=\rho^3(B_0)$, we have $\rho^3(B_0)\in K_{\boldsymbol{111}}$.
  Now, since $\rho^2(B_0)\neq B_0$, there is some $y\in \rho^2(B_0)\setminus B_0$. By conjugating $y$ by $a$, we obtain that $a\in K_{\boldsymbol{111}}$
  and therefore $\langle a, \rho^3(B_0)\rangle \leq K_{\boldsymbol{111}}$.
  Since $\rho$ is cyclic, we may repeat the above procedure until we reach the first $n$ such that $\rho^n(B_0)=B_0$
  at which point we have $\langle a, \rho^n(B_0)\rangle =\langle a, B_0\rangle \leq K_{\boldsymbol{1}^n}$.


  \emph{Case $m=2$:} This case includes only the Grigorchuk--Erschler (with polynomial $x^2+1$) and Grigorchuk (with polynomial $x^2+x+1$) groups.
  First note that $K=\langle [a, b_1]\rangle^G$. For all $n\geq 1$ and for all $b\in B\setminus B_0$ (in other words, for $b=b_1$ or $b=b_0b_1$), we have
  $\varphi_{\boldsymbol{1}}([a,b]^{2n})=[a, \rho(b)]^{n}$.  
  Let $k$ be the smallest integer such that $\rho^k(b_1)=b_0$ (such an integer exists, since $b_1=\rho(b_0)$ and $\rho$ is of finite order). 
  Then, by induction, we get
  \[\varphi_{\boldsymbol{1}^k}([a,b_1]^{2^{k}})= [a,b_0].\]
  Hence,
  \[\varphi_{\boldsymbol{1}^{k+1}}([a,b_1]^{2^{k}}) = \varphi_{\boldsymbol{1}}([a,b_0])= b_1.\]
  It follows that $b_0 = \varphi_v([a,b_1]^{2^{k}})$ where $v=\boldsymbol{1}^{2k+1}$.
 
 Now, suppose there exists $g\in G$ which maps $u=\boldsymbol{1}^{2k}\boldsymbol{0}$ to $v$ and such that $\varphi_v(g)=1$. 
  Then, since $a=\varphi_{u}([a,b_1]^{2^{k}})$, writing $h=[a,b_1]^{2^{k}}$ we have $\varphi_v(h^g)=\varphi_{u}(h)^{\varphi_v(g)}=a$. 
  Thus $K_v\geq \langle a, b_0\rangle$ as $K$ is normal in $G$. 
  
  Now, for the Grigorchuk--Erschler group, $k=1$ and we may take $g=b_0^{b_1^a}$.
  For the Grigorchuk group, $k=2$ and we take $g=b_1^{b_0^f}$ where $f=b_1^{b_1^a}$.
%
%
%
%
%
%
%
%
%
   \end{proof}

 \begin{lem}
 Let $G\in \mathcal{G}_{2,m}$ with $m\geq 2$. 
  The derived subgroup $K'$ of $K\leq G$ contains $\St_G(n+m+2)$ where $n$ is as in Lemma \ref{K>B0}.
 \end{lem}
\begin{proof}

We first show that $\psi([K,\langle a, B_0\rangle])\geq K\times K$.
Let $b\in B_1$, then  putting 
\[x:=[[b_{m-1},a],\rho^{-1}(b)]\in [K,\langle a, B_0\rangle]\]
we have
\begin{align*}
\psi(x) &= [\psi([b_{m-1},a]), \psi(\rho^{-1}(b))]\\
&=([a\rho(b_{m-1}), 1], [\rho(b_{m-1})a, b])\\
&=(1, a\rho(b_{m-1})b\rho(b_{m-1})ab)\\
&= (1, [a,b]).
\end{align*}
Thus $$\psi([K,\langle a, B_0\rangle])\geq 1\times K.$$ 
Since $a$ normalizes $[K,\langle a, B_0\rangle]$, we also have $\psi([K,\langle a, B_0\rangle])\geq K\times 1$, which proves our claim.

By Lemma \ref{K>B0}, there exists $n$ such that $K_v \geq \langle a, B_0\rangle$ where $v=\boldsymbol{1}^n$. 
 The fact that $G$ is regular branch over $K$ implies that 
 $$\psi_n(K\cap \St_G(n)) \geq 1\times \dots \times 1\times K $$
  where there are $2^n$ factors in the direct  product. 
 Thus, 
 $$\psi_n(K'\cap\St_G(n)) \geq 1\times \dots \times 1\times [K,\langle a, B_0\rangle].$$  
This, together with the first claim, implies that 
$$\psi_{n+1}(K'\cap\St_G(n+1))\geq 1\times \dots \times 1\times K\times K$$
where there are $2^{n+1}$ factors in the product.
Since $K'\cap\St_G(n+1)$ is normal in $G$, which acts transitively on the $(n+1)$th level, we obtain  that
$\psi_{n+1}(K'\cap\St_G(n+1))$ contains a direct product of $2^{n+1}$ conjugates of $K$ which are actually just $K$. 
 Lemma 9 of \cite{Sunic} states that $\St_G(m+1)\leq K$, so 
 \begin{align*}
\psi_{n+1}(K'\cap\St_G(n+1)) &\geq \St_G(m+1)\times \dots\times \St_G(m+1)\\
 & =\psi_{n+1}(\St_G(n+m+2)).  
 \end{align*}

 This yields that $K'\geq \St_G(n+m+2)$, because $\psi_{n+1}$ is injective.
\end{proof}


\section{Dense subgroups in the Aut T topology}\label{sec:dense}
The following proposition is due to P.-H. Leemann. 
It gives a method for finding dense subgroups of a countable group in the $\Aut T$ topology. 

\begin{prop}\label{PH}
Let $T$ be the $d$-regular rooted tree for some $d\geq 2$.
 Let $G\leq \Aut T$ be countably generated by $S=\{g_1, g_2,\dots\}$.
Suppose that $m_1,m_2,\ldots, \in\NN$ are coprime with $|G/\St_G(n)|$ for all $n\in\NN$. 
Then  $H:= \langle g_1^{m_1}, g_2^{m_2}, \dots\rangle$ is a dense subgroup of $G$ with respect to the $\Aut T$ topology. 
%
\end{prop}
\begin{proof}
Recall that the basic open sets of the $\Aut T$ topology for $G$ are the cosets of $\St_G(n)$ for $n\in\NN$. 
Pick $g\in S$ and $m\in \NN$  coprime with $f_n:=|G/St_G(n)|$ for all $n\in\NN$. 
Then, for each $n\in\NN$, the Euclidean algorithm yields $x, y\in\ZZ$ such that $xm+yf_n=1$. 
Thus we can write  	$g=(g^m)^x(g^{f_n})^y$. 
Since $g^{f_n}\in\St_G(n)$, we have $g\St_G(n)=(g^m)^x\St_G(n)$ and therefore $\langle g\rangle \St_G(n)=\langle g^m\rangle \St_G(n)$.
This implies that $G=\langle g_1,g_2,\dots\rangle \St_G(n)=\langle g_1^{m_1},g_2^{m_2},\dots\rangle \St_G(n)$ for every $n$ 
and so $H$ is dense in $G$ with respect to the $\Aut T$ topology. 
%
\end{proof}

\begin{cor}\label{cor:dense}
 If $G=\langle a, B\rangle \in \mathcal{G}\setminus \mathcal{G}_{2,1}$ then, for any $q$ coprime to $p$ and any $b\in B$,
 the subgroup $H(q)=\langle (ab)^q, B\rangle$ is dense in $G$ for the $\Aut T$ topology. 
 Since $G$ has the congruence subgroup property, the $\Aut T$ topology coincides with the profinite topology, so $H(q)$ is dense with respect to the profinite topology.
\end{cor}

Of course, $H(q)$ could well be the whole of $G$ 
(this will indeed be the case for the Grigorchuk group for example, by Pervova's result \cite{Pervova}). 
We will show in the next section that for every non-torsion $G\in\mathcal{G}_{2,m}$ where $m\geq 2$,
each $H(q)$ is in fact a proper subgroup of $G$ when we take $b$ to be the generator such that $b=(a,b)$.
This makes the next proposition more interesting.

\begin{prop}\label{prop:HqConjugatetoG}
Let $G\in\mathcal{G}_{2,m}$ be a non-torsion \v{S}uni\'{c} group acting on the binary tree and $q\geq 3$ be an odd number.
Let $b\in B$ such that $\psi(b)=(a,b)$.
The subgroup $H(q)=\langle (ab)^q, B\rangle$ is conjugate to $G$ in $\Aut T$.
\end{prop}
\begin{proof}
The existence of $b$ as in the statement is ensured by Corollary \ref{cor:contain_dihedral}.
Let $g\in \Aut T$ be defined by $\psi(g) = ((ba)^{\frac{q-1}{2}}g, g)$. 
We have
\begin{equation*}
g^{-1}(ab)^qbg = g^{-1}(ab)^{q-1}agaa
\end{equation*}
and
\begin{align*}
\psi(g^{-1}(ab)^{q-1}aga) &= \psi(g^{-1})\psi((abab)^{\frac{q-1}{2}})\psi(aga) \\
&=(g^{-1}(ab)^{\frac{q-1}{2}}, g^{-1})((ba)^{\frac{q-1}{2}}, (ab)^{\frac{q-1}{2}})(g,(ba)^{\frac{q-1}{2}}g)\\
&=(1,1).
\end{align*}
Since $\psi$ is injective, this means that $g^{-1}(ab)^qbg=a$.

Now, let $x\in B$. Then,
\begin{align*}
\psi(g^{-1}xg) &= (g^{-1}(ab)^{\frac{q-1}{2}}, g^{-1})(\omega(x),\rho(x))((ba)^{\frac{q-1}{2}}g, g)\\
&=(g^{-1}(ab)^{\frac{q-1}{2}}\omega(x)(ba)^{\frac{q-1}{2}}g, g^{-1}\rho(x)g) \\
&=(\omega(x), g^{-1}\rho(x)g).
\end{align*}
Indeed, if $\omega(x)=1$, then $g^{-1}(ab)^{\frac{q-1}{2}}\omega(x)(ba)^{\frac{q-1}{2}}g = 1$, and if $\omega(x)=a$, then
\begin{equation*}
g^{-1}(ab)^{\frac{q-1}{2}}\omega(x)(ba)^{\frac{q-1}{2}}g = g^{-1}(ab)^{q}bg = a = \omega(x).
\end{equation*}
This inductively implies that the action of $g^{-1}xg$ on $T_2$ is the same as the action of $x$, so $g^{-1}xg = x$, as $G$ acts faithfully on $T_2$ by definition.
Since $H$ is generated by $(ab)^{q}b$ and $B$, we conclude that $g^{-1}Hg = G$.
\end{proof}

In fact this also shows that $H(q)$ is dense in $G$:
it suffices to show that $\tfrac{H(q)\St(n)}{\St(n)}=\tfrac{G\St(n)}{\St(n)}$ for every $n\in \NN$. 
Fixing $n\in \NN$, let us use \, $\bar{ }$ \, to denote the images modulo $\St(n)$. 
Then $\overline{G}=\overline{H^g}=\overline{H}^{\overline{g}}\leq \overline{G}^{\overline{g}}$, where $g\in \Aut T$ is the element constructed in the proof of Proposition \ref{prop:HqConjugatetoG}.
Since these quotients are finite, we have equality throughout in this expression and therefore $\overline{G}=\overline{H}$, as required.

\section{Proper dense subgroups}\label{sec:proper}
Let $\widetilde{\mathcal{G}}_{2,m}$ denote the non-torsion groups in $\mathcal{G}_{2,m}$. 
By Corollary \ref{cor:contain_dihedral}, this is precisely the family of groups in $\mathcal{G}_{2,m}$
which contain a copy of the infinite dihedral group generated by $a$ and $b=(a,b)\in B$.
The main result in this section is
%
%

\begin{thm}\label{thm:proper}
 Let  $G\in\widetilde{\mathcal{G}}_{2,m}$ with $m\geq 2$.
 Then for each odd $q\in\NN$, the subgroup $H(q):=\langle (ab)^q, B\rangle$ is proper and dense in the profinite topology. 
\end{thm}

Notice that if two odd numbers $q_1,q_2\in \NN$ are relatively prime, then $H(q_1)$ and $H(q_2)$ are contained in different maximal subgroups 
(necessarily of infinite index, since $H(q_1)$ and $H(q_2)$ are dense).
Indeed, if they were contained in the same maximal subgroup $M$, this $M$ would contain $ab$ and $B$, so we would have $M=G$, a contradiction. 
%
Thus, Theorem \ref{thm:proper} immediately implies:
 \begin{cor}\label{cor:AtLeastCountablyMaximals}
  Let  $G\in\widetilde{\mathcal{G}}_{2,m}$ with $m\geq 2$. 
  Then $G$ contains at least countably many maximal subgroups of infinite index. 
 \end{cor}

The main idea of the proof of the theorem is as follows: we have already seen in the last section that $H(q)$ is a dense subgroup of $G$, so it suffices to show that it is proper, indeed that it does not contain $ab$. 
This will be done by comparing the actions of $G$ and $H(q)$ on the boundary $\partial T$ of the tree;
that is, by comparing the Schreier (orbital) graphs $G/\St_G(\xi)$ and $H(q)/\St_{H(q)}(\xi)$ for $\xi \in \partial T$ 
and proving that no word in the generators of $H(q)$ can produce the same action as $ab$ on $\partial T$.
In fact, it will suffice to do this for a particular point $\xi\in \partial T$, the ray labeled by $\tilde{\boldsymbol{1}}=\boldsymbol{1}^{\infty}$, an infinite string of $\boldsymbol{1}$s. 




Recall that the binary tree can be seen as the set $\{\boldsymbol{0},\boldsymbol{1}\}^{*}$ of finite words over $\{\boldsymbol{0},\boldsymbol{1}\}$, and there is a natural addition operation on this alphabet.
 The action of $G$ on $\{\boldsymbol{0},\boldsymbol{1}\}^{*}$ is described in terms of the action of the generators. 
Let $s\in \{\boldsymbol{0},\boldsymbol{1}\}^{*}$ and $s_1, s_2\in \{\boldsymbol{0},\boldsymbol{1}\}$. The action of $a\in G$ is given by
\[a\cdot s_1s = (s_1+\boldsymbol{1})s.\]
The action of $x\in B$ is recursively defined by $x(s_1) = s_1$ and
\[x\cdot s_1s_2s = 
\begin{cases}
\boldsymbol{0}s_2 s & \text{ if } s_1 = \boldsymbol{0} \text{ and } \omega(x) = 1\\
\boldsymbol{0}(s_2+1)s & \text{ if } s_1=\boldsymbol{0}  \text{ and } \omega(x) = a\\
\boldsymbol{1}\rho(x)(s_2s) & \text{ if } s_1 = \boldsymbol{1}
\end{cases}
\]
where $\omega$ and $\rho$ are as in Definition \ref{def:SunicGroups}.

The boundary of the tree $\partial T$ can be identified with the set $\{\boldsymbol{0},\boldsymbol{1}\}^{\NN}$ of right-infinite strings of $\boldsymbol{0}$s and $\boldsymbol{1}$s.
The action of $G$ on $T$ can naturally be extended to an action of $G$ on $\partial T$ by
taking $s\in\{\boldsymbol{0},\boldsymbol{1}\}^{\NN}$ in the above definition. 

Let us make some useful remarks regarding the action of $G$ on $\partial T$.

\begin{rem}\label{rem:ActionOfb}
Let $b=(a,b)\in G$. 
Then, for $s\in \{\boldsymbol{0},\boldsymbol{1}\}^{\NN}$, it follows from the definition that $b\cdot s$ is the string obtained by adding $\boldsymbol{1}$ (modulo 2) to the element immediately following the first $\boldsymbol{0}$ in the sequence.
\end{rem}
\begin{rem}\label{rem:ActionsOfBAndb}
For $x\in B$ and $s\in \{\boldsymbol{0},\boldsymbol{1}\}^{\NN}$, either $x\cdot s = s$ or $x\cdot s = b\cdot s$.
It follows that for any $\xi\in \partial T$, the orbit of $\xi$ under the action of $G$ is the same as the orbit of $\xi$ under the action of $\langle a,b\rangle \leq G$.
\end{rem}

The action of $G$ on $\partial T$ cannot be transitive, since $G$ is finitely generated and $\partial T$ is uncountable.
We will therefore restrict our attention to the action of $G$ on the orbit of $\tilde{\boldsymbol{1}}=\boldsymbol{1}^{\infty} \in \partial T$. According to Remark \ref{rem:ActionsOfBAndb}, this orbit is the same as the orbit of $\tilde{\boldsymbol{1}}$ under $\langle a,b\rangle \cong D_\infty$. 
Let us study the orbital graph of the action of $\langle a,b\rangle$ on this orbit.

\begin{prop}\label{prop:StructureOfOrbitalGraph}
The orbital graph of the action of $\langle a,b\rangle$ on the orbit of $\tilde{\boldsymbol{1}}$ is a half-line with a loop labeled by $b$ at $\tilde{\boldsymbol{1}}$.

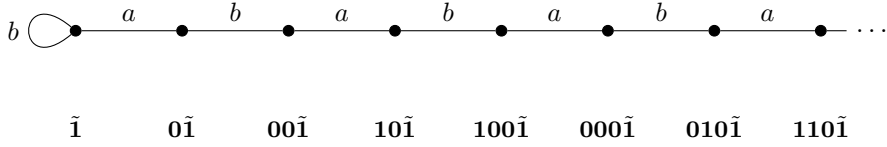
\begin{tikzpicture}[scale=0.7]
\begin{scope}

\begin{scope}
\tikzstyle{every node}=[circle, draw, fill, inner sep=0, minimum size = 4]
\node (V1) at (-7,0) {};
\node (V2) at (-5,0) {};
\node (V3) at (-3,0) {};
\node (V4) at (-1,0) {};
\node (V5) at (1,0) {};
\node (V6) at (3,0) {};
\node (V7) at (5,0) {};
\node (V8) at (7,0) {};
\end{scope}
\node (V9) at (8,0) {$\dots$};

\def\b{6}
\def\c{1}
\draw (V1)node[below = \c]{\small{$\tilde{\boldsymbol{1}}$}} to node[above, midway]{$a$}
(V2)node[below = \c]{\small{$\boldsymbol{0}\tilde{\boldsymbol{1}}$}} to node[above,midway]{$b$}
(V3)node[below = \c]{\small{$\boldsymbol{00}\tilde{\boldsymbol{1}}$}} to node[above, midway]{$a$}
(V4)node[below = \c]{\small{$\boldsymbol{10}\tilde{\boldsymbol{1}}$}} to node[above,midway]{$b$}
(V5)node[below = \c]{\small{$\boldsymbol{100}\tilde{\boldsymbol{1}}$}} to node[above, midway]{$a$}
(V6)node[below = \c]{\small{$\boldsymbol{000}\tilde{\boldsymbol{1}}$}} to node[above,midway]{$b$} 
(V7)node[below = \c]{\small{$\boldsymbol{010}\tilde{\boldsymbol{1}}$}} to node[above, midway]{$a$} 
(V8)node[below = \c]{\small{$\boldsymbol{110}\tilde{\boldsymbol{1}}$}} to (V9);

\draw 
(V1) to[out=140, in=220, looseness =25] node[left]{$b$}(V1);
\end{scope}
\end{tikzpicture}

\end{prop}
\begin{proof}
Since $\langle a,b \rangle$ is generated by two elements of order $2$, the orbital graph must be a connected graph where every vertex has degree $2$ (where a loop adds only $1$ to the degree of the vertex).
Hence, there are only four possibilities : the graph is either a circle, a line, a segment or a half-line.




It follows from the definition of the action that $b\cdot s\ne s$ for all $s\in \{\boldsymbol{0},\boldsymbol{1}\}^{\NN}$ containing at least one $\boldsymbol{0}$ (see Remark \ref{rem:ActionOfb}) and that $a\cdot s\ne s$ for all $s\in \{\boldsymbol{0},\boldsymbol{1}\}^{\NN}$. 
Hence, $b$ has exactly one fixed point, $\tilde{\boldsymbol{1}}$, and $a$ has none.
Therefore, by degree considerations, the only possible orbital graph of the action of $\langle a,b\rangle$ on the orbit of $\tilde{\boldsymbol{1}}$ must be a half-line with a loop at $\tilde{\boldsymbol{1}}$ labeled by $b$.
\end{proof}


Proposition \ref{prop:StructureOfOrbitalGraph} allows us to define a bijection between the orbit of $\tilde{\boldsymbol{1}}$ and $\ZZ$.

\begin{prop}\label{prop:BijectionZOrbits}
There is a bijection $\zeta$ between $\ZZ$ and the orbit $G\cdot\tilde{\boldsymbol{1}}$ of $\tilde{\boldsymbol{1}}$ under $G$ given by $\zeta(n)=(ab)^n\cdot\tilde{\boldsymbol{1}}$.
\end{prop}
\begin{proof}

According to Remark \ref{rem:ActionsOfBAndb}, for every $\xi \in G\cdot\tilde{\boldsymbol{1}}$, there exists $g\in \langle a,b \rangle$ such that $\xi = g\cdot \tilde{\boldsymbol{1}}$. 
Since $\langle a,b\rangle$ is isomorphic to the infinite dihedral group, there exists $n\in \ZZ$, $m\in \{0,1\}$ such that $g = (ab)^n b^m$. 
Since $b\cdot \tilde{\boldsymbol{1}} = \tilde{\boldsymbol{1}}$, we conclude that $\zeta$ is surjective.

Injectivity follows from Proposition \ref{prop:StructureOfOrbitalGraph}.
Indeed, it follows from the structure of the orbital graph (see the figure in that proposition) that $(ab)^n\cdot \tilde{\boldsymbol{1}}$ is the vertex at distance $2n-1$ from $\tilde{\boldsymbol{1}}$ if $n>0$ 
and at distance $2n$ if $n\leq 0$.
\end{proof}

The bijection $\zeta$ allows us to define an action of $G$ on $\ZZ$.
It turns out that the restriction of this action to $\langle a,b\rangle$ is the standard action of $D_\infty$ on $\ZZ$. 
This will allow us to prove Theorem \ref{thm:proper}.

\begin{proof}[Proof of Theorem \ref{thm:proper}]
Let $\zeta \colon \ZZ \rightarrow G\cdot\tilde{\boldsymbol{1}}$ be the map from Proposition \ref{prop:BijectionZOrbits}. 
Then, $G$ acts on $\ZZ$ by
\[g\cdot n =\zeta^{-1}(g \cdot \zeta (n))\]
for all $g \in G$ and $n\in \NN$. In particular,
\[(ab)^q\cdot n = \zeta^{-1}((ab)^{q+n}\cdot \tilde{\boldsymbol{1}}) = q+n\]
and
\[b\cdot n = \zeta^{-1}(b(ab)^{n}\cdot \tilde{\boldsymbol{1}}) = \zeta^{-1}((ab)^{-n}b \cdot \tilde{\boldsymbol{1}}) = \zeta^{-1}((ab)^{-n} \cdot \tilde{\boldsymbol{1}}) = -n\]
for all $n\in \ZZ$.
According to Remark \ref{rem:ActionsOfBAndb}, for $x\in B$, we also have $x\cdot n = \pm n$, depending on the value of $n$.

Since $H(q)$ is generated by $(ab)^q$ and $B$, it follows that $H(q)\cdot 0 = q\ZZ \subsetneq \ZZ = G\cdot 0$. Therefore, $H(q) \neq G$.
\end{proof}

\section{Projections of proper dense subgroups}\label{sec:Pervova}

In this section, we show that for suitable self-replicating groups $G$, the projection of a proper dense subgroup to any vertex is also a proper dense subgroup. 
These results are generalizations of Lemmas 4.1 and 4.4 of \cite{Pervova} to a larger class of branch groups which includes \v{S}uni\'{c} groups.

\begin{rem}
Let $G$ be any group and $M\leq G$ a dense subgroup in the profinite topology. 
Then $M\cap H$ is dense in the profinite topology of $H$ whenever $H$ is a finite index subgroup of $G$. 
This reduces to observing that every basic open set in the profinite topology of $H$ contains a basic open set in the profinite topology of $G$. 
To wit, if $K\lhd H$ has finite index, then it also has finite index in $G$ (but is not necessarily normal);
and its normal core in $G$ is a normal subgroup of finite index in $G$ contained in $K$.
\end{rem}

\begin{lem}\label{lem:MdenseActsLikeG}
If $G$ is a group of automorphisms of a regular rooted tree and $M\leq G$ is dense in the profinite topology, then 
\begin{enumerate}[label=(\roman*)]
	\item $M$ has the same action as $G$ on all truncated trees and in particular on all levels of the tree;
	\item $M_u\leq G_u$ is dense for every $u\in T$.
\end{enumerate}
\end{lem}
\begin{proof}
The first item follows by observing that the action of $M$ on a truncated tree (or indeed on level $n$) is given by $M/\St_M(n)\cong M\St_G(n)/\St_G(n)=G/\St_G(n)$.

The second item follows by noting that $\St_M(u)$ is dense in $\St_G(u)$ for every $u\in T$ (by the above remark). 
Since $\varphi_u:\St_G(u)\rightarrow G_u$ is onto, the correspondence theorem yields that $M_uN_u=G_u$ for any $N_u\lhd G_u$ of finite index. 
\end{proof}

\begin{prop}\label{prop:PervovaProperProjection}
Let $G$ be a just infinite branch group, which is self-replicating and acts regularly and primitively on the first level of the tree (that is, as a cyclic group of prime order $p$). 
If $M<G$ is a proper dense subgroup of $G$ then so is $M_u<G_u=G$ for every vertex $u\in T$. 
\end{prop}
\begin{proof}
Suppose for a contradiction that $M_u=G_u=G$ for some $u\in T$ and suppose furthermore that $u$ is of smallest level possible. 
Then, by Lemma \ref{lem:MdenseActsLikeG}, we can assume that $u\in X^1$.
Let $R$ denote the restriction of $\rist_M(u)$ to $T_u$. 
Then $R\lhd M_u=G_u=G$. 
Since $G$ is just infinite, either $R$ has finite index in $G$ or it is trivial. 
If the former holds, then $\rist_M(v)_v$ has finite index in $G_v=G$ for every $v\in X$ and, in particular, 
$\prod_{X}\rist_M(v)$ has finite index in $G$, which contradicts the fact that $M$ has infinite index in $G$. 
Thus $R=1$ and therefore $\rist_M(v)=1$ for every $v\in X$. 

We claim that $\psi(\St_M(1))$ is a diagonal subgroup of $G\times \dots \times G$. 
The fact that $G$ (and therefore $M$, by Lemma \ref{lem:MdenseActsLikeG}) acts regularly on $X$ implies that $\St_M(v)=\St_M(1)$ for every $v\in X$ and therefore $\psi(\St_M(1))$ is subdirect in $G\times\dots\times G$, as $\St_M(1)_v =\St_M(v)_v=M_v=G$ for every $v\in X$. 
Supposing  that for any proper subset $V\subset X$  
\begin{equation}\label{eqn:rstMtrivial}
M\cap \prod_{v\in V}\rist(v)=M\cap \prod_{v\in V}\Aut T=\ker (M\rightarrow \prod_{X\setminus V}\Aut T)=1
\end{equation}
then successive applications of Goursat's lemma show that 
\begin{equation}\label{eqn:MisDiagonal}
\psi(\St_M(1))=\{(g,\alpha_2(g),\dots,\alpha_p(g)) \mid g\in G, \alpha_2,\dots,\alpha_p\in \Aut G \}.
\end{equation}
To show that \eqref{eqn:rstMtrivial} holds, we proceed by induction on the size of $V\subset X$.
The base case $|V|=1$ having been shown above, assume that \eqref{eqn:rstMtrivial} holds for all subsets of $X$ of size at most $n\in\{1,\dots,p-2\}$ and let $V\subset X$ have size $n+1$. 
Suppose for a contradiction that $K:=M\cap \prod_{v\in V}\rist(v)\neq 1$, so there exists $v\in V$ such that $K_v\neq 1$. 
Since $K_v\lhd \St_M(1)_v=G_v=G$ and $G$ is just infinite, $K_v$ is of finite index in $G$. 
Now, because $M$ acts primitively on $X$, no non-trivial partition of $X$ is preserved by $M$ and, in particular, there exists $m\in M$ such that $0<|V\cap V^m|<|V|$.
Thus 
$$[K,K^m]\leq M\cap\prod_{v\in V\cap V^m}\rist(v)=1$$
by inductive hypothesis
and $K_v$ commutes with $(K^m)_{v^m}\lhd \St_M(1)_{v^m}=G$. 
Therefore the intersection $K_v\cap (K^m)_{v^m}$ is an abelian normal subgroup of finite index in $G$, 
which contradicts the assumption that $G$ is a branch group (as these cannot be virtually abelian, as can be checked by considering rigid stabilisers). 
Thus $K=1$ and the claim follows by induction. 

As stated above, this claim allows us to use Goursat's lemma successively to conclude \eqref{eqn:MisDiagonal}.
Now, since $G$ is a branch group, $\rist_G(v)$ is non-trivial for every $v \in T$ and therefore
$\rist_G(v)_v$ is a non-trivial normal subgroup of $\St_G(1)_v=G$ for every $v\in X$, which implies that it is of finite index.

Since the action of $G$ on the first level is cyclic of order $p$ and $G$ is self-replicating,
for each $n$, the quotient $G/\St_G(n)$ is a subgroup of the iterated wreath product $C_p\wr\dots\wr C_p$ of $C_p$ with itself $n$ times. 
In particular, $G$ is residually nilpotent and therefore $\rist_G(\boldsymbol{0})_{\boldsymbol{0}}$ is not contained in the $k$th term $\gamma_k(G)$ of the lower central series of $G$, for some $k\in \NN$. 
Let $g\in \rist_G(\boldsymbol{0})$ with $\psi(g)=(r,1,\dots,1)$ and $r\in\rist_G(\boldsymbol{0})_{\boldsymbol{0}}\setminus \gamma_k(G)$. 
Now, since $G$ is just infinite, $N:=G\cap \psi^{-1}(\gamma_k(G)\times\dots\times\gamma_k(G))$ has finite index in $G$
and as $M$ is dense, $MN=G$. 
This means that $g=mn$ for some $m=(m_1,\alpha_2(m_1),\dots,\alpha_p(m_1))\in \St_M(1)$,  $n\in N$. 
We thus have 
$$\psi(g)=(r,1,\dots,1)=(m_1n_1,\alpha_2(m_1)n_2,\dots,\alpha_p(m_1)n_p)$$
where $n_1,\dots,n_p\in \gamma_k(G)$. 
But this implies that $\alpha_i(m_1)\in\gamma_k(G)$ for all $i$, hence $m_1$ must be in $\gamma_k(G)$, making $r$ an element of $\gamma_k(G)$, a contradiction.
\end{proof}

\section{Finitely generated maximal subgroups of infinite index}\label{sec:maximal}

Here we prove that the proper subgroups $H(q)$ of Section \ref{sec:proper} are in fact maximal when $q$ is an odd prime.
The strategy is to show, using an argument about length reduction, that for any $g\notin H(q)$, there exists a vertex $v$ such that the projection of $\langle H(q), g\rangle$ to $v$ is not proper. This, in turn, implies that $\langle H(q), g \rangle$ cannot be proper by Proposition \ref{prop:PervovaProperProjection}.

We will use the homomorphism $\phi$ defined in Proposition \ref{prop:liftIsHomomorphism} to show that the subgroups from Theorem \ref{thm:proper} are indeed maximal. 
We start with some more auxiliary results. 

\medskip
\noindent\textbf{Notation.} 
Henceforth $G$ will denote a group in $\widetilde{\mathcal{G}}_{2,m}$ with $m\geq 2$, which contains, by Corollary \ref{cor:contain_dihedral} an element $b\in B$ such that $\psi(b)=(a,b)$ and, by Lemma \ref{lemma7} $c\in B_{-1}\setminus B_0$ and $d\in B_0\setminus B_1$ such that $c=(a,d)$. 
Unless stated otherwise, $q\geq 3$ will be a fixed odd integer (not necessarily prime) and $H:=H(q)=\langle (ab)^q, B\rangle$.

\begin{lem}\label{lemma:IsomorphismWithdAndb}
Consider the following subgroups of $G$:
\begin{align*}
\Delta_b = \langle a, d^{(ab)^{\frac{q-1}{2}}} \rangle, &&
\Delta_d = \langle a, d^{(ad)^{\frac{q-1}{2}}} \rangle.
\end{align*}
There is a unique isomorphism $f\colon \Delta_b \to \Delta_d$ such that $f(a) = a$ and $f(d^{(ab)^{\frac{q-1}{2}}}) = d^{(ad)^{\frac{q-1}{2}}}$.
\end{lem}

\begin{proof}
Let $D_{2\cdot 4} = \langle s,t \mid s^2=t^2=(st)^4 = 1\rangle$ be the dihedral group of order 8.
 For all $g\in G$, we have
\begin{equation*}
a^2 = \left(d^g\right)^2 = 1
\end{equation*}
and
\begin{align*}
\psi((ad^g)^4) &= \psi(ad^gad^g)^2 \\
&=(\rho(d)^{g'}, \rho(d)^{g'})^2 \\
&=(1,1)
\end{align*}
(for some $g'\in G$), which means that $(ad^g)^4=1$. It follows that there are unique homomorphisms $g_b\colon D_{2\cdot 4} \to \Delta_b$ and $g_d\colon D_{2\cdot 4} \to \Delta_d$ such that
\begin{align*}
g_b(s)&=a \\ 
g_b(t) &= d^{(ab)^{\frac{q-1}{2}}}\\ 
g_d(s)&=a\\
g_d(t) &= d^{(ad)^{\frac{q-1}{2}}}.
\end{align*}
These homomorphisms are clearly surjective, and a direct computation shows that they are injective. Therefore, we have the isomorphism $f=g_d\circ g_b^{-1}$.
\end{proof}

%
%

\begin{lem}\label{lemma:GeneratorsOfStabilizerOfH}
The stabilizer $\St_H(1)$ is generated by $\{x, x^{a(ba)^{q-1}} \mid x\in B \}$
\end{lem}
\begin{proof}
Clearly, $ \{x, x^{a(ba)^{q-1}} \mid x\in B \}$ generates a subgroup of $\St_H(1)$.
On the other hand, if $h\in \St_H(1)$, then it can be written as a product $h=h_1h_2\dots h_n$, with $h_i\in \{a(ba)^{q-1}\}\cup B$  (since this set generates $H$). 
To act trivially on the first level, this product must contain an even number of $a$, and therefore an even number of $a(ba)^{q-1}$. 
Since $(a(ba)^{q-1})^{-1} = a(ba)^{q-1}$, this implies that $h$ is indeed in the subgroup generated by  $\{x, x^{a(ba)^{q-1}} \mid x\in B \}$.
\end{proof}

\begin{prop}\label{prop:H1SubdirectProductOfH}
We have $\psi(\St_H(1))\leq H^{(ab)^{\frac{q-1}{2}}} \times H$.
Furthermore, the projection of $\psi(\St_H(1))$ on each factor is surjective.
\end{prop}
\begin{proof}
For all $x\in B$, we have
\begin{equation*}
\psi(x) = (\omega(x), \rho(x)).
\end{equation*}
If $\omega(x) = 1$, then $\omega(x)$ is clearly in $H^{(ab)^{\frac{q-1}{2}}}$. Otherwise,
\begin{equation*}
\omega(x) = a = (ba)^{\frac{q-1}{2}}(ab)^qb(ab)^{\frac{q-1}{2}} \in H^{(ab)^{\frac{q-1}{2}}}.
\end{equation*}
Moreover, $\rho(x)\in B \subset H$, so $\psi(x)\in H^{(ab)^\frac{q-1}{2}}\times H$.
Similarly,
\begin{equation*}
\psi(x^{a(ba)^{q-1}}) = (\rho(x)^{(ab)^{\frac{q-1}{2}}}, \omega(x)^{(ba)^{\frac{q-1}{2}}}) \in H^{(ab)^{\frac{q-1}{2}}} \times H.
\end{equation*}
The first result then follows from the fact that $\St_H(1)$ is generated by the elements of $B$ and their conjugates by $a(ba)^{q-1}$.


Now, for all $x\in B$, we have $\rho^{-1}(x) \in \St_H(1)$ with $\psi(\rho^{-1}(x)) = (\omega(\rho^{-1}(x)), x)$. Since we also have $(ab)^{2q}\in \St_H(1)$ with $\psi((ab)^{2q}) = ((ba)^q, (ab)^q)$, we see that the projection of $\psi(\St_H(1))$ on the second factor is surjective. To see that the projection on the first factor is also surjective, it suffices to notice that for all $h\in \St_H(1)$ with $\psi(h) = (h_1,h_2)$, we have $h^{a(ba)^{q-1}} \in \St_H(1)$ with $\psi(h^{a(ba)^{q-1}}) = (h_2^{(ab)^{\frac{q-1}{2}}}, h_1^{(ba)^{\frac{q-1}{2}}})$.
\end{proof}

Recall that in Proposition \ref{prop:liftIsHomomorphism} we proved the existence and uniqueness of the homomorphism $\phi\colon G\rightarrow \St_G(1)$ such that $\phi(a)=aca$ and $\phi(x)=\rho^{-1}(x)$ for all $x\in B$. 

\begin{prop}\label{prop:ProjectionToTrivialAndHMeansInH}
 If there exists $g\in \St_G(1)$ such that $\psi(g) = (1,h)$ for some $h\in H$, then $g\in \St_H(1)$.
\end{prop}
\begin{proof}
Since $h\in H$, there exist $h_1,h_2\dots, h_n \in \{a(ba)^{q-1}\} \cup B$ such that 
\begin{equation*}
h = h_1h_2\dots h_n.
\end{equation*}
For $1\leq i \leq n$,  define
\begin{equation*}
\tilde{h}_i := \begin{cases}\rho^{-1}(h_i) & \text{ if } h_i\in B \\ c^{a(ba)^{q-1}} & \text{ if } h_i = a(ba)^{q-1}\end{cases}
\end{equation*}
and $\tilde{h} = \tilde{h}_1\tilde{h}_2\dots \tilde{h}_n$. 
Each of the terms in the product is in $\St_H(1)$, therefore so is $\tilde{h}$.
Notice that $\psi(\tilde{h_i}) = (x_i, h_i)$, for $1\leq i\leq n$,  where $x_i\in \{1,a,d^{(ab)^{\frac{q-1}{2}}}\}$. Therefore, writing $x=x_1x_2\dots x_n$, we have
\begin{equation*}
\psi(\tilde{h}) = (x, h).
\end{equation*}

On the other hand, by direct computation, we see that $\psi(\phi(h_i)) = (f(x_i), h_i)$, where $f$ is the isomorphism of Lemma \ref{lemma:IsomorphismWithdAndb}. Since $\phi$ is a homomorphism (by Proposition \ref{prop:liftIsHomomorphism}), we have $\phi(h) = \phi(h_1)\phi(h_2)\dots\phi(h_n)$. Hence,
\begin{align*}
\psi(\phi(h)) = (f(x), h).
\end{align*}
However, by Proposition \ref{prop:LiftIsUnique}, since $\psi(g) = (1,h)$, we must have $g = \phi(h)$. Hence, $f(x) = 1$. Since $f$ is an isomorphism, this means that $x=1$. Therefore,
\begin{equation*}
\psi(\tilde{h}) = (1,h) = \psi(g).
\end{equation*}
This implies that $g = \tilde{h}\in \St_H(1)$.
\end{proof}

\begin{cor}\label{cor:ProjectionsInHMeansInH}
If $g\in \St_G(1)$ is such that $\psi(g) = (g_0,g_1)$ with $g_0\in H^{(ab)^{\frac{q-1}{2}}}$ and $g_1\in H$, then $g\in \St_H(1)$.
\end{cor}
\begin{proof}
By Proposition \ref{prop:H1SubdirectProductOfH}, there exists $h\in \St_H(1)$ such that $\psi(h) = (g_0, h_1)$ for some $h_1\in H$. Hence,
\begin{equation*}
\psi(gh^{-1}) = (1,g_1h_1^{-1}).
\end{equation*}
By Proposition \ref{prop:ProjectionToTrivialAndHMeansInH}, this means that $gh^{-1}\in \St_H(1)$, so $g\in \St_H(1)$.
\end{proof}

\begin{notation}
For any $g\in \St_G(1)$, we will write
\begin{equation*}
\lambda(g) = \text{min}\left\{l(g') \mid g=\gamma g'\delta, \quad \gamma,\delta \in \St_{\langle a, b\rangle}(1), g'\in \St_G(1)\right\},
\end{equation*}
where $l(g')$ is the length of $g'$ in $G$ with respect to the generating set $\{a\}\cup B\setminus\{1\}$.
\end{notation}


\begin{lem}\label{lem:HqReducedLength}
If $Q\leq G$ contains $H=H(q)$ properly then there exist $n\in \NN$ and $s \in \St_G(1)\setminus H$ with $\lambda(s) \leq 3$ such that $\langle s, H \rangle \leq Q_{\boldsymbol{1}^n}$.
\end{lem}
\begin{proof}
By assumption, $Q \ne H$, so there exists some $g\in Q\setminus H$. 
Replacing $g$ by $g(ab)^q$ if necessary, we may assume that $g\in\St_Q(1)$.

Let $\gamma, \delta\in\St_{\langle a, b\rangle}(1)$ and $g'\in\St_G(1)$ be such that $g=\gamma g'\delta$ and $\lambda(g)= l(g')$. We have
\[\psi(g)=(\gamma_0g'_0\delta_0, \gamma_1g'_1\delta_1)\]
where $\psi(\gamma)=(\gamma_0,\gamma_1), \psi(g')=(g'_0,g'_1), \psi(\delta)=(\delta_0,\delta_1)$. 
Note that $\gamma_0,\gamma_1, \delta_0,\delta_1 \in\langle a, b\rangle$ as $\delta,\gamma\in\St_{\langle a, b \rangle}(1)$.

If $\gamma_1g'_1\delta_1\in H$ then, Corollary \ref{cor:ProjectionsInHMeansInH} implies that $(ab)^{\frac{q-1}{2}}\gamma_0 g'_0 \delta_0(ba)^{\frac{q-1}{2}}\notin H$. 
Thus, replacing $g$ by $g^{(ab)^qb}$ if necessary, we may assume that $\varphi_{\boldsymbol{1}}(g)=\gamma_1g'_1\delta_1\notin H$. 

If $\varphi_{\boldsymbol{1}}(g)\notin \St_G(1)$ then  $\varphi_{\boldsymbol{1}}(g(ab)^{2q})=\varphi_{\boldsymbol{1}}(g)(ab)^q\in\St_G(1)$ so, replacing $g$ by $g(ab)^{2q}$ we can suppose that $\varphi_{\boldsymbol{1}}(g)=\gamma_1g'_1\delta_1\in \St_G(1)$. 

If $\gamma_1\notin\St(1)$ then $\varphi_{\boldsymbol{1}}((ab)^{2q}g(ba)^{2q})=(ab)^q\gamma_1g'_1\delta_1(ba)^q$ with $(ab)^q\gamma_1\in\St_{\langle a,b\rangle}(1)$. 
So, replacing $g$ by $(ab)^{2q}g(ba)^{2q}$ if needed, we have
\[\varphi_{\boldsymbol{1}}(g)=\gamma_1g'_1\delta_1\in \St_G(1)\setminus H \textrm{ with } \gamma_1, \delta_1\in\langle a, b\rangle, \textrm{ and } \gamma_1\in\St(1).\]

Now, if $\delta_1\in\St(1)$ then $\lambda(\varphi_{\boldsymbol{1}}(g))\leq l(g'_1)$; otherwise $g'_1a\in\St_G(1)$ and $a\delta_1\in\St_{\langle a, b\rangle}(1)$ so $\lambda(\varphi_{\boldsymbol{1}}(g))\leq l(g'_1)+1$.
Remark \ref{rem:Contraction} implies that $l(g'_1) \leq \frac{l(g')+1}{2}$, and  $l(g') = \lambda(g)$ by construction. Hence
\[\lambda(\varphi_{\boldsymbol{1}}(g)) \leq \frac{\lambda(g)+3}{2}.\]
By repeating this procedure (which we can do thanks to Proposition \ref{prop:H1SubdirectProductOfH}), we conclude by induction that there exists some $n\in \NN$ and $y\in (Q\setminus H)\cap \St_G(\boldsymbol{1}^n)$ such that $s:= \varphi_{\boldsymbol{1}^n}(y) \notin H$ and $\lambda(s) \leq 3$.
\end{proof}

\begin{lem}\label{lem:HqContainsNotYetEqualsDivisor}
If $Q\leq G$ contains $H=H(q)$ properly then there exists $m\in \NN$ such that 
$Q_{\boldsymbol{1}^m}\geq\langle (ab)^r, B\rangle=H(r)$ for some proper divisor $r$ of $q$.
\end{lem}
\begin{proof}
By Lemma \ref{lem:HqReducedLength}, there exist $n\in \NN$ and $s\in \St_G(1)\setminus H$ with $\lambda(s) \leq 3$ such that $\langle s, H \rangle \leq Q_{\boldsymbol{1}^n}$.
 Thus, it suffices to show the result for $Q=\langle g, H \rangle$ for some $g\in \St_G(1)$ such that $\lambda(g) \leq 3$, which we do below in several cases.

\textbf{The case $\lambda(g)=0$: } In this case, $g\in \langle a,b \rangle$.
 We can therefore assume that $g=(ab)^{k}$ for some $k\in \mathbb{Z}$ (multiplying on the left or on the right by $b\in H$ if necessary). 
Since $k$ cannot be a multiple of $q$, there exist $i,j\in \ZZ$ such that $ik+jq=r$, where $r$ is the greatest common divisor of $k$ and $q$.
Hence $(ab)^r=(ab)^{ij}(ab)^{jq}\in \langle g,H\rangle$ and so $Q=\langle g,H\rangle \geq H(r)$.

\textbf{The case $g=(ab)^{-k}x(ab)^k$ for $x\in B\setminus\{1,b\}$ and $k\in \mathbb{Z}$: } Conjugating by an appropriate power of $(ab)^q$ if necessary, we can assume that $k$ is a positive odd number.
Note that $k$ cannot be a multiple of $q$ as $g\notin H$.

%

Then 
\[\psi(g)=((ab)^{\frac{k-1}{2}}a\rho(x)a(ba)^{\frac{k-1}{2}},(ba)^{\frac{k-1}{2}}b\omega(x)b(ab)^{\frac{k-1}{2}}).\]
If $\omega(x)=a$, then the second coordinate in the above expression is $(ba)^kb$.
 Since $H\leq Q_{\boldsymbol{1}}$ by Proposition \ref{prop:H1SubdirectProductOfH}, we have that $Q_{\boldsymbol{1}}$ contains $(ab)^k$
and therefore also $H(r)$, where $r<q$ is the greatest common divisor of $q$ and $k$, by the same argument as in the previous case.
If $\omega(x)=1$, then consider $(ba)^qg(ab)^q$ instead of $g$. 
Its image under $\varphi_{\boldsymbol{1}}$ is $(ba)^{\frac{q+k}{2}}\rho(x)(ab)^{\frac{q+k}{2}}$ where $\frac{q+k}{2}$ cannot be a multiple of $q$ (this is guaranteed by Corollary \ref{cor:ProjectionsInHMeansInH}).
Since $H_{\boldsymbol{1}}=H$ by Proposition \ref{prop:H1SubdirectProductOfH}, we may take $(ba)^{\frac{q+k}{2}}\rho(x)(ab)^{\frac{q+k}{2}}$ as our new $g$ and repeat this case. 
By Proposition \ref{prop:SunicProp2}, there exists some minimal $m\in \NN^*$ such that $\omega(\rho^{m-1}(x))=a$.
 Therefore, by repeating the above procedure $m-1$ times, we get that $Q_{\boldsymbol{1}^m}\geq H(r)$ for some proper divisor $r$ of $q$.

\textbf{The case $\lambda(g) = 1$: } The fact that $g = \gamma g' \delta$ with $\gamma, \delta \in \St_{\langle a, b \rangle}(1)$ and $g'\in \St_G(1)$ with $l(g')=1$
immediately implies that $g'\in B$, $\gamma = [a](ba)^{k'_1}[b]$ and $\delta = [b](ab)^{k'_2}[a]$ for some $k'_1,k'_2\in\NN$ (where the square brackets mean that an element might not be present).
Multiplying $\gamma$ by $b$ on the left and $\delta$ by $b$ on the right if necessary and using the fact that $\gamma,\delta \in \St_G(1)$, we can assume that $\gamma = (ba)^{2k_1}[b]$ and $\delta = [b](ab)^{2k_2}$ for some $k_1,k_2\in \NN$. Therefore,
 \begin{equation}\label{eqn:casel=1}
g = (ba)^{2k_1}[b]g'[b](ab)^{2k_2}
 = (ba)^{2k_1}x(ab)^{2k_2}
\end{equation}
for some $x\in B\setminus\{1,b\}$.

If $\omega(x)=a$ then consider $(ba)^qg(ab)^q \in Q$:
\begin{align*}
\varphi_{\boldsymbol{1}}((ba)^qg(ab)^q ) &=(ba)^{\frac{q-1}{2}+k_1}b\omega(x)b(ab)^{\frac{q-1}{2}+k_2}\\
						&=(ba)^{\frac{q-1}{2}+k_1}bab(ab)^{\frac{q-1}{2}+k_2}\\
						&=(ba)^{\frac{q+1}{2}+k_1}(ba)^{\frac{q-1}{2}+k_2}b = (ba)^{q+k_1+k_2}b.
\end{align*}
As $H\leq Q_{\boldsymbol{1}}$ by Proposition \ref{prop:H1SubdirectProductOfH}, we get that $Q_{\boldsymbol{1}}$ contains $(ba)^{k_1+k_2}$.

If $k_1+k_2$ is not a multiple of $q$, then $Q_{\boldsymbol{1}}\geq H(r)$ by a previously considered case, where $r$ is the greatest common divisor of $q$ and $k_1+k_2$.

If $k_1+k_2=nq$ for some $n\in \NN$ then $g=(ba)^{2nq-2k_2}x(ab)^{2k_2}$.
Hence
\begin{align*}
b(ba)^{-2nq}g &= b(ba)^{-2k_2}x(ab)^{2k_2} \\
&=b(ab)^{2k_2}x(ab)^{2k_2}\\
&=(ba)^{2k_2}(bx)(ab)^{2k_2}.
\end{align*}
As $bx\in B\setminus\{1,b\}$, this is a previously considered case, so we know that there exists $m\in \NN$ such that $Q_{\boldsymbol{1}^m}\geq H(r)$
for some proper divisor $r$ of $q$.

If $\omega(x)=1$, then $\varphi_{\boldsymbol{1}}(g)=(ba)^{k_1}\rho(x)(ab)^{k_2} \notin H$ can be made of the same form as \eqref{eqn:casel=1} by multiplying on the left and right by an appropriate power of $(ab)^q$. 
We can therefore repeat the argument explained above. 
This will terminate after a finite number of repetitions as there is some $j\in\NN$ such that $\omega(\rho^j(x))=a$.
%
%
This finishes the case $\lambda(g)=1$.

\textbf{The case $\lambda(g)=2$: } Elements in $G$ of length 2 are of the form $ax$ or $xa$ for some $x\in B$.
 None of these elements are in $\St_G(1)$, so this case does not arise.

\textbf{The case $\lambda(g) = 3$: } If $l(g')=3$ in the minimal decomposition $g = \gamma g' \delta$, then $g'=axa$ or $g'=yaz$ for some $x,y,z \in B\setminus\{1,b\}$.
 However, $yaz \notin \St_G(1)$, so this case is impossible. 
 Hence, $g' = axa$, so $g=\gamma axa \delta$ with $\gamma,\delta \in \St_{\langle a,b\rangle }(1)$. 
 We have
\begin{align*}
(ab)^{-q}g(ab)^q &= (ab)^{-q}\gamma axa \delta (ab)^{q} \\
&= ((ab)^{-q}\gamma a)x(a \delta (ab)^{q}) \\
&= \gamma' x \delta'
\end{align*}
with $\gamma' = (ab)^{-q}\gamma a$ and $\delta' = a\delta (ab)^{q} \in \St_{\langle a,b\rangle }(1)$.
 This means that $\lambda((ab)^{-q}g(ab)^q) \leq 1$, so from what we have already shown, we conclude that 
there exist $m\in \NN$ and a proper divisor $r$ of $q$ such that $Q_{\boldsymbol{1}^m}\geq H(r)$.
\end{proof}

\begin{lem}\label{lem:HqcontainsDivisors}
If $Q\leq G$ contains $H=H(q)$ properly then there exists $n\in \NN$ such that 
$Q_{\boldsymbol{1}^n}=H(t)$ for some proper divisor $t$ of $q$. 
\end{lem}

\begin{proof}
By Lemma \ref{lem:HqContainsNotYetEqualsDivisor}, there exist $m\in \NN$ and a proper divisor $r$ of $q$ such that $Q_{\boldsymbol{1}^m}\geq H(r)$.
 If $Q_{\boldsymbol{1}^m} = H(r)$, then we are done.
  Otherwise, by applying Lemma \ref{lem:HqContainsNotYetEqualsDivisor} to $Q_{\boldsymbol{1}^m}$, we find $m'\in \NN$ and a proper divisor $s$ of $r$ such that $Q_{\boldsymbol{1}^{m+m'}}\geq H(s)$.

As $q$ only has a finite number of divisors, repeating this procedure as often as necessary, we will find some $n\in \NN$ such that either $Q_{\boldsymbol{1}^n} = H(t)$ for some proper divisor $t\geq 3$ of $q$, or $Q_{\boldsymbol{1}^n}\geq H(1)$. Since $H(1)=G$ and $Q_{\boldsymbol{1}^n} \leq G$, the latter case yields $Q_{\boldsymbol{1}^n}=H(1)$.
\end{proof}

\begin{thm}\label{thm:maximal_proof}
For every odd prime $q$ the subgroup $H(q)<G$ is maximal. 
\end{thm}
\begin{proof}
By Theorem \ref{thm:justinfinite}, and the general properties of {\v{S}}uni{\'c} groups quoted in Section \ref{sec:prelims}, 
$G$ satisfies all the hypotheses of Proposition \ref{prop:PervovaProperProjection}. 
Therefore if $M<G$ is a proper subgroup, dense in the profinite topology, then so is $M_v<G_v=G$ for every $v\in T$. 
Fix an odd prime $q$ and denote $H(q)$ by $H$. 
Recall that by Theorem \ref{thm:proper}, $H$ is proper and dense in the profinite topology.
Since $H$ is dense, so is $\langle g, H\rangle$ for any $g\in G\setminus H$.
Now,  Lemma \ref{lem:HqcontainsDivisors} states that there exists $n\in \NN$
such that $(\langle g, H\rangle)_{\boldsymbol{1}^n}= H(t)$, where $t$ is a proper divisor of $q$.
Since $q$ is prime, we get $(\langle g, H\rangle)_{\boldsymbol{1}^n}= H(1)=G$. Hence, by Proposition \ref{prop:PervovaProperProjection}, $\langle g, H\rangle=G$. We conclude that $H$ is maximal.
\end{proof}

\section{Finding all maximal subgroups of infinite index}\label{sec:countably}

In this section, we show that for any $G\in \widetilde{\mathcal{G}}_{2,m}$ the subgroups $H(q)=\langle (ab)^q, B\rangle$ with $q$ an odd prime  are the only maximal subgroups of infinite index, up to conjugation.
Let us first state the theorem and its proof, assuming some auxiliary results that will be shown below. 
The notation is the same as in the previous section. 

\begin{thm}\label{thm:allmaxareconjtoH}
	Let $M<G$ be a maximal subgroup of infinite index. 
	Then there exist an odd prime $q$ and $g\in G$ such that $M=H(q)^g$.
\end{thm}
\begin{proof}
	By Lemma \ref{lem:maximalprojectstoHq}, there exist $v\in T$ of level, say, $n$ and an odd prime $q$ such that $M_v=H(q)=:H$. 
Because $M$ is a maximal subgroup of infinite index, it is dense. As $G$ is level-transitive, Lemma \ref{lem:MdenseActsLikeG} implies that $M$ acts transitively on $X^n$.
	In particular, for each $w\in X^n$ there exists $m\in M$ taking $w$ to $v$. 
	Writing $m=\sigma\mu$ with $\sigma\in\Sym X^n$, $\mu\in\St(n)$ and $\varphi_w(\mu)=\mu_w\in G$, we have
	\[M_w=\varphi_w(\St_M(w))=\varphi_w(m^{-1}\St_M(v)m)=\mu_w^{-1}M_v\mu_w=H^{\mu_w}.\]
	Lemma \ref{lem:nprojtoanyconjofH} yields some $g\in G$ such that $(H^g)_w=H^{\mu_w}$ for each $w\in X^n$. 
	Since $H^g$ is also maximal of infinite index, Lemma \ref{lem:maximalswithsameprojsareequal} ensures that $M=H^g$, as required. 
\end{proof}

We start by showing that maximal subgroups of infinite index are determined by their projections.

\begin{lem}\label{lem:StabilizerOfMaximalSubgroups}
	Let $M<G$ be a maximal subgroup. 
	Then $$\St_M(n)=\bigcap_{w\in X^n}(\varphi_w^{-1}(M_w)\cap G)$$
	for each $n\in \NN.$
\end{lem}
\begin{proof}
	Fix $n\in \NN$ and put $J:=\bigcap_{w\in X^n}(\varphi_w^{-1}(M_w)\cap G)$. 
	Then $\St_M(n)\leq J\leq\St_G(n)$. 
	Suppose for a contradiction that there exists $g\in J\setminus \St_M(n)$.
	Note that $g\notin M$ because $g\in \St_G(n)$, so $\langle M, g\rangle =G$ and $(\langle M,g\rangle)_w=G$ for all $w\in X^n$. 
	We will obtain a contradiction by showing that $\varphi_w(\alpha)\in M_w$ for all $w\in X^n$ and $\alpha\in\St(w)\cap\langle M,g\rangle$. 
	Fix such $w$, $\alpha$ and write
	$$\alpha=m_1g^{\epsilon_1}m_2\cdots m_kg^{\epsilon_k}=(\mu_1^{-1}g^{\epsilon_1}\mu_1) (\mu_2^{-1}g^{\epsilon_2}\mu_2)\cdots (\mu_k^{-1}g^{\epsilon_k}\mu_k)\mu_k^{-1}$$
	for some $k\in \NN$, $m_i\in\NN$, $\epsilon_i\in\{-1,0,1\}$  and $\mu_i=m_i^{-1}\cdots m_1^{-1}$ for $i=1,\dots, k$. 
	Note that $\alpha\in\St(w)$ implies that $\mu_k^{-1}\in\St_M(w)$, as $g\in\St(n)$, so $\varphi_w(\mu_k^{-1})\in M_w$. 
	If $\varphi_w(m^{-1}gm)\in M_w$ for all $m\in M_w$, then $\varphi_w(\alpha)$ is a product of elements of $M_w$, as required. 
	
	To show this last claim, fix $m\in M$ and suppose it maps $w$ to $u$. 
	By definition of $J$, there exists $h\in\St_M(u)$ such that $\varphi_u(g)=\varphi_u(h)$. 
	Writing $m=\tau\mu$ with $\tau\in\Sym X^n$ and $\mu\in\St(n)$, we have 
	$$\varphi_w(m^{-1}gm)=\varphi_w(\mu^{-1})\varphi_u(g)\varphi_w(\mu)=\varphi_w(\mu^{-1})\varphi_u(h)\varphi_w(\mu)=\varphi_w(m^{-1}hm)$$
	where $\varphi_w(m^{-1}hm)\in M_w$ because $h\in\St_M(u)$ and $(\St_M(u))^m=\St_M(w)$. 
\end{proof}

\begin{lem}\label{lem:maximalswithsameprojsareequal}
	Let $L, M <G$ be two maximal subgroups of infinite index. 
	If there exists $n\in\NN$ such that $L_w=M_w$ for each $w\in X^n$, then $L=M$. 
\end{lem}
\begin{proof}
	By Lemma \ref{lem:StabilizerOfMaximalSubgroups}, $S:=\St_{L}(n)=\St_{M}(n)$ and this subgroup is normal in $L$ and $M$ but not in $G$, because $G$ is just infinite. 
	So the normaliser $N_G(S)$ of $S$ in $G$ is a proper subgroup of $G$ containing $L$ and $M$, forcing $L=N_G(S)=M$, by the maximality assumption. 
\end{proof}

By Proposition \ref{prop:H1SubdirectProductOfH}, all projections of $H(q)$ are conjugates of $H(q)$. 
We now reverse-engineer this and show that for any collection of conjugates of $H(q)$ there is a conjugate of $H(q)$ whose projections are precisely this collection.

\begin{lem}\label{lem:subdirectwithHonright}
	For any odd prime $q$ and any $g\in G$ there exist $s\in \St_G(1)$ and $h_0,h_1\in H(q)$ such that $s=(gh_0,h_1)$. 
\end{lem}
\begin{proof}
	For any $g\in G$ we can find $s_1\in\St_G(1)$ such that $s_1=(g,y)$ with $y\in \langle a, b\rangle$ by writing $g$ as a word in $a\cup B\setminus\{1\}$ and replacing each instance of $a$ by $b$ and each $x\in B\setminus\{1\}$ by $a\rho^{-1}(x)a$.
	
	Since $\langle a, b\rangle \cong D_{\infty}$, there exists $l\in \ZZ$ such that either $y=(ab)^l$ or $y=(ab)^la$. 
	Because $q$ is coprime to 4, there exist $m,n\in \ZZ$ such that $4m+l=qn$. 
	Put $$s_2:=\begin{cases}
	(acab)^{4m}=((da)^{4m}, (ab)^4m)=(1,(ab)^{4m}) & \text{ if } y=(ab)^{l}\\
	aba(acab)^{4m}=(b,a(ab)^{4m}) & \text{ if } y=(ab)^la
	\end{cases}$$
	so that $s:=s_1s_2=\begin{cases}
	(g,y(ab)^{4m})=(g,(ab)^{qn}) & \text{ if } y=(ab)^l\\
	(gb,ya(ab)^{4m})=(gb,(ab)^{qn}) & \text{ if } y=(ab)^la.
	\end{cases}$
\end{proof}

\begin{lem}\label{lem:1projtoanyconjugateofH}
	Let $q$ be an odd prime and $H:=H(q)$. 
	For any $g_0, g_1\in G$ there exists $s\in G$ such that 
	\[\varphi_{\boldsymbol{0}}(\St_{H^s}(1))=H^{g_0} \text{ and } \varphi_{\boldsymbol{1}}(\St_{H^s}(1))=H^{g_1}.\]
\end{lem}
\begin{proof}
	Since $G$ is self-replicating, there exist $s_1\in\St_G(1)$ and $y\in G$ such that $s_1=(y,g_1)$. 
	Proposition \ref{prop:H1SubdirectProductOfH} states that $\varphi_{\boldsymbol{0}}(\St_{H}(1))=H^{(ab)^{(q-1)/2}}$ and $\varphi_{\boldsymbol{1}}(\St_{H}(1))=H$. 
	Apply Lemma \ref{lem:subdirectwithHonright} to $(ba)^{(q-1)/2}$ and to $yg_0^{-1}$ to obtain $s_2\in\St_G(1)$ and $h_0,h_1,k_0,k_1\in H$ such that
	$s_2=((ba)^{(q-1)/2}h_0k_0^{-1}g_0y^{-1}, h_1k_1^{-1})$. 
	Putting $s:=s_2s_1$ yields the result. 
\end{proof}

\begin{lem}\label{lem:nprojtoanyconjofH}
	Let $q$ be an odd prime and $H:=H(q)$. 
	For every $n\in \NN$ and any set $\{g_v\}_{v\in X^n}\subset G$ there exists $s\in G$ such that $(H^s)_v=H^{g_v}$. 
\end{lem}
\begin{proof}
	We induct on $n$. 
	The case $n=0$ is trivial and the case $n=1$ is Lemma \ref{lem:1projtoanyconjugateofH}. 
	So suppose the claim true for $n\geq 1$ and pick $\{g_v\}_{v\in X^{n+1}}\subset G$. 
	Write each $v\in X^{n+1}$ as $v=\boldsymbol{i}w$ with $\boldsymbol{i}\in X$ and $w\in X^n$. 
	For each $\boldsymbol{i}\in X$, the inductive hypothesis implies that there exists $t_{\boldsymbol{i}}\in G$ such that for each $w\in X^n$, $(H^{t_{\boldsymbol{i}}})_w=H^{g_{\boldsymbol{i}w}} $.
	Lemma \ref{lem:1projtoanyconjugateofH} yields some $s\in G$ such that 
	$(H^s)_{\boldsymbol{i}}=H^{t_{\boldsymbol{i}}}$
	for each $\boldsymbol{i}\in X$. 
	Thus for each $v\in X^{n+1}$ we obtain, 
	$	(H^s)_v =((H^s)_{\boldsymbol{i}})_w=(H^{t_{\boldsymbol{i}}})_w 
	= H^{g_{\boldsymbol{i}w}}=H^{g_v}$.
	%
\end{proof}

The last remaining ingredient in the proof of Theorem \ref{thm:allmaxareconjtoH} is showing that all maximal subgroups of infinite index project to some $H(q)$. 
For this, we first show in Lemma \ref{prop:DenseSubgroupsProjectToSomeHq} that for any proper and dense subgroup  $M<G$, there exists $v\in T$ and an odd $l>1$ such that the projection $M_v$ is equal to $H(l)$. 
This uses similar techniques to \cite{Pervova}, namely that  $\varphi_{\boldsymbol{1}^n}((baz)^{2^n})$ for $z\in G'$ reduces in length as $n$ grows. 
In the setting of \cite{Pervova}, where the groups are torsion, this length eventually reaches 1 and the generator $a$ is obtained. 
In our setting, the length might stabilise, for example if we started off with $(ab)^k$ (see Lemma \ref{lemma:PowersOf(ab)AreForever}). However, we show in Lemma \ref{lemma:ThetaMapConvergence} that this is essentially the only possibility. 
Thus any dense subgroup projects to $(ab)^k$ for some odd $k$ (Lemma \ref{lemma:HContainsPowerOf(ab)}). 
With a bit more work and carefully using the contraction properties of $G$,  we obtain in Proposition \ref{prop:DenseSubgroupsProjectToSomeHq} that any proper dense subgroup projects to some $H(q)$ for some odd $q$. 
It is then a small step in Lemma \ref{lem:maximalprojectstoHq} to show that if the proper dense subgroup  $M$ is maximal then this $q$ is prime. 
%
%
%

\begin{lem}\label{lemma:ClassesOfG'}
If $z\in G'$, then $\varphi_{\boldsymbol{0}}(z) \equiv \varphi_{\boldsymbol{1}}(z) \equiv y$ modulo $G'$, where $y\in B_1 \cup abB_1$.
\end{lem}
\begin{proof}
As $G'$ is generated by conjugates of $[a,x]$ with $x\in B$ and  
\begin{equation*}
\psi([a,x]) = (\rho(x)\omega(x), \omega(x)\rho(x)),
\end{equation*}
we have $\varphi_{\boldsymbol{0}}([a,x])\equiv \varphi_{\boldsymbol{1}}([a,x])$ modulo $G'$.
 If $x\in B_0$, then $\varphi_{\boldsymbol{1}}([a,x]) = \rho(x) \in B_1$. If $x\notin B_0$, then $x'=bx\in B_0$ and $x=bx'$, so $\varphi_{\boldsymbol{1}}([a,x]) = \varphi_{\boldsymbol{1}}([a,bx']) = ab\rho(x') \in abB_1$.

Since conjugating $[a,x]$ by an element of $G$ conjugates and possibly permutes the projections $\varphi_{\boldsymbol{0}}([a,x])$ and $\varphi_{\boldsymbol{1}}([a,x])$, the result is true for the generators of $G'$. 
Hence, since $\left(B_1\cup abB_1\right)G'$ is a subgroup of $G/G'$, the result is also true for any $z\in G'$.
\end{proof}

%
\begin{defn}
	Define $\Theta\colon G' \rightarrow G'$ by $\Theta(z) = a\varphi_{\boldsymbol{0}}(z)a\varphi_{\boldsymbol{1}}(z)$.
	This is well-defined by the previous lemma.
\end{defn}
%
%
 
\begin{notation}
For $z\in G$, write $|z|$ for the $B$-length of $z$, that is, the minimal number of $B$-letters required to represent $z$ as a word in the alphabet $\{a\}\cup B\setminus\{1\}$.
\end{notation}

Note that, similarly to Remark \ref{rem:Contraction},   $|z|\geq |\varphi_{\boldsymbol{0}}(z)| + |\varphi_{\boldsymbol{1}}(z)|$ for any $z\in G'$. 
It follows that $|\Theta(z)| \leq |z|$.

\begin{lem}\label{lemma:ThetaMapConvergence}
Let $z\in G'$ be such that $|\Theta^n(z)|=|z|$ for all $n\in \NN$.
 If $|z|\geq 3$, then there exists $l\in \NN$, $x\in B$ such that
\begin{equation*}
z=axa(ba)^{2l}x.
\end{equation*}
Otherwise, either $z=1$ or $|z|=2$ and there exists $x\in B\setminus\{1\}$ such that $z=axax$ or $z=xaxa$.
\end{lem}
\begin{proof}
%
If $|z|=0$ then $z=1$, because $z\in G'$. 
Proposition \ref{prop:AbelianisationGIsAB} implies that $|z|=1$ cannot hold and that if $|z|=2$ there must exist $x\in B\setminus\{1\}$ with $z=axax$ or $z=xaxa$.

Assume that $|z| = k$ with $k\geq 3$. 
Then there exist $x_1,\dots x_k\in B\setminus\{1\}$ such that $z$ can be written in one of four forms:
\begin{center}
\begin{tabular}{lcl}
(1) $x_1a\dots ax_k$ & or &(2) $ax_1\dots x_ka$ if $k$ is odd, \\
(3) $ax_1\dots ax_k$ & or & (4) $x_1a\dots x_ka$ if $k$ is even.\\
\end{tabular}
\end{center}

If $z$ is of form (1) then 
$$\psi(z) = (\omega(x_1)\rho(x_2)\dots \omega(x_k), \rho(x_1)\omega(x_2)\dots \rho(x_k))$$
so
$$\Theta(z) = a\omega(x_1)\rho(x_2)\dots\omega(x_k)a\rho(x_1)\omega(x_2)\dots\rho(x_k).$$
Since $|\Theta(z)|=|z|$, no cancellation occurs in the expression for $\Theta(z)$, except possibly the term $a\omega(x_1)$. 
This means that $\omega(x_k)=1$ and $\omega(x_i)=a$ for $1<i<k$. 
But then $a\omega(x_1)=1$ because $\Theta(z)\in \St_G(1)$ and so $a$ must occur an even number of times in the expression for $\Theta(z)$.
Hence $\Theta(z)$ is of the same form as $z$, with last letter $\rho(x_k)$. 
Then, since $|\Theta^{n}(z)|=|z|$ for all $n$, by induction we obtain that $\rho^n(x_k)\in\ker\omega$ for all $n$, which implies that $x_k$ is trivial. 
So form (1) cannot occur.

If $z$ is of form (2), noticing that this is a conjugate of form (1) by $a$, we obtain that 
$$\Theta(z)=a\rho(x_1)\omega(x_2)\dots\rho(x_k)a\omega(x_1)\rho(x_2)\dots\omega(x_k).$$
Similar arguments as in case (1) show that $x_1$ is trivial, another contradiction. 

Form (3) yields 
$$\Theta(z)= a\rho(x_1)\omega(x_2)\dots \omega(x_k)a \omega(x_1)\rho(x_2)\dots \rho(x_k).$$
Again, $ |\Theta(z)| = |z|$ implies that there is no cancellation and therefore $\omega(x_k)=\omega(x_1)$ and 
$\Theta(z)$ is of the same form as $z$. 
Since $|\Theta^n(z)|=|z|$ for all $n\in \NN$, we obtain by induction that $\omega(\rho^n(x_i))=a$ for $1<i<k$ and $\omega(\rho^n(x_1))=\omega(\rho^n(x_k))$ for all $n\in \NN$. 
This means that $x_i=b$ for $1<i<k$ and $x_1=x_k=x\in B\setminus\{1\}$. 
So $z=axa(ba)^{2l}x$ where $2l=k-2$, for some $x\in B\setminus\{1\}$.

If $z$ is of form (4) then 
\[\Theta(z) = a\omega(x_1)\rho(x_2)\dots \rho(x_k)a\rho(x_1)\omega(x_2)\dots\omega(x_k).\]
To avoid length reduction and cancellation we must have $\omega(x_i)=a$ for $1<i<k$
and $\Theta(z)\in \St_G(1)$ implies that $\omega(x_1)=\omega(x_k)$. 
If $\omega(x_1)=\omega(x_k)=1$ then $\Theta(z)$ is of form (3). 
The same argument as in the case (3) then implies that $\rho(x_1)=b$ and thus $\omega(x_1)=a$, a contradiction. 
So $\omega(x_1)=\omega(x_k)=a$ and therefore $\Theta(z)$ has the same form as $z$. 
Repeating the argument for $\Theta(z)$ instead of $z$, we obtain that $\omega(\rho(x_2))=\omega(\rho(x_{k-1}))=a$
and $\omega(\rho(x_i))=a$ also for all the other $i$. 
Hence, by induction, $\omega(\rho^n(x_i))=a$ for $1\leq i\leq k$ and all $n\in \NN$. 
In other words, $z=(ba)^{2l}$ where $k=2l$. 
\end{proof}

\begin{lem}\label{lemma:HContainsPowerOf(ab)}
Let $M\leq G$ be a dense subgroup.
 There exists an odd $k\in \ZZ$ and a vertex $v\in T$ such that the projection $M_v$ contains $(ab)^k$.
\end{lem}
\begin{proof}
Since $M$ is dense, $MG'=G$ because $G'$ is of finite index in $G$, so there exists $z\in G'$ such that $baz\in M$.
 We have
$$\varphi_{\boldsymbol{1}}((baz)^2) = b\varphi_{\boldsymbol{0}}(z)a\varphi_{\boldsymbol{1}}(z) = ba\Theta(z),$$
so $\varphi_{\boldsymbol{1}^n}((baz)^{2^n})  = ba\Theta^n(z)$ 
for all $n\in \NN$. 
Since $\left\{|\Theta^n(z)|\right\}_{n\in \NN}$ is a decreasing sequence, it must eventually become constant.
 Hence, it follows from Lemma \ref{lemma:ThetaMapConvergence} that there exist $l\in \NN$, $x\in B$ and a vertex $v\in T$ such that the projection $M_v$ contains either $ba(axa(ba)^{2l}x)$ or $ba(xaxa)$.
 In the former case,
\begin{align*}
\varphi_{\boldsymbol{0}}((baaxa(ba)^{2l}x)^2) &= \varphi_{\boldsymbol{0}}(x(ba)^{4l+2}x)\\
&= \omega(x)(ab)^{2l+1}\omega(x)
=\begin{cases}
(ba)^{2l+1} & \text{if } \omega(x)=a \\
   (ab)^{2l+1} & \text{if } \omega(x) =1.
\end{cases}
\end{align*}

If instead $ba(xaxa)\in M_v$, then either $x=b$, in which case we obtain the desired result, or $x\ne b$.
 In the latter case, we can assume without loss of generality that $\omega(x)=1$.
  Indeed, if $\omega(x)=a$, then 
$\varphi_{\boldsymbol{1}}((baxaxa)^2) = ba\rho(x)a\rho(x)a$
so we can repeat this until $\rho^i(x)\in\ker\omega$, at which point $\varphi_{\boldsymbol{0}}((ba\rho^i(x)a\rho^i(x)a)^2) = ab.$
\end{proof}

\begin{lem}\label{lemma:HContainsElementOfB}
Let $M\leq G$ be a dense subgroup.
 Then, there exist $\beta\in B\setminus\{1\}$ and a vertex $v\in T$ such that $\beta\in M_v$.
\end{lem}
\begin{proof}
We will show that for all $x\in B\setminus \{1\}$ and all $z\in G'$, there exist some vertex $v\in T$ and some element $\beta\in B\setminus \{1\}$ such that $\beta\in \langle xz \rangle_v$.
 The result will then follow, since by the density of $M$, for any $x\in B\setminus\{1\}$, there exists some $z\in G'$ such that $xz\in M$.
 Recall from Remark \ref{rem:Contraction} that $l(\varphi_v(g))\leq (l(g)+1)/2$ for $g\in\St_G(1), v\in X$ where $l$ is the usual word metric.

For $x\in B\setminus\{1\}$ and $z\in G'$, we have
\begin{equation*}
\psi(xz) = (\omega(x)\varphi_{\boldsymbol{0}}(z), \rho(x)\varphi_{\boldsymbol{1}}(z)).
\end{equation*}
By Lemma \ref{lemma:ClassesOfG'}, there exists $y\in B_1$ such that $\varphi_{\boldsymbol{0}}(z)\equiv_{G'} \varphi_{\boldsymbol{1}}(z)$ are both congruent modulo $G'$ to  $y$ or $aby$.
This gives four cases:

\begin{center}
\begin{tabular}{lr}
	(1)(i) $x\notin B_0$, $\varphi_{\boldsymbol{0}}(z)\equiv_{G'} y$  & (1)(ii) $x\notin B_0$, $\varphi_{\boldsymbol{0}}(z)\equiv_{G'} aby$\\
	(2)(i) $x\in B_0$, $\varphi_{\boldsymbol{0}}(z)\equiv_{G'} y$ & (2)(ii) $x\in B_0$, $\varphi_{\boldsymbol{0}}(z)\equiv_{G'} aby$
\end{tabular}
\end{center}
For case (1)(i), $\omega(x)=a$ and $\rho(x)\notin B_1$, so $x':=\rho(x)y\neq 1$ since $y\in B_1$.
Thus there exists $z'\in G'$ such that $\varphi_{\boldsymbol{1}}(xz)=x'z'$ and $l(x'z')=l(\varphi_{\boldsymbol{1}}(xz))\leq \frac{l(xz)+1}{2}$.

For case (1)(ii), note that $b\notin B_1$ because $b=(a,b)$ so $by\neq 1$ and there exists $z'\in G'$ such that
$\varphi_{\boldsymbol{0}}(xz)=x'z'$ 
where $x'=\omega(x)aby=by\in B\setminus\{1\}$ and $l(x'z')=l(\varphi_{\boldsymbol{1}}(xz))\leq \frac{l(xz)+1}{2}$.

In case (2)(i),  $\varphi_{\boldsymbol{0}}(xz)=yz'$ and 
$\varphi_{\boldsymbol{1}}(xz)=\rho(x)yz''$ for some $z', z''\in G'$.
If $y\neq 1$, take the former; the latter otherwise.
Either way, $l(\varphi_v(xz))\leq\frac{l(xz)+1}{2}$ for $v\in X$.

In case (2)(ii), $\varphi_{\boldsymbol{0}}(xz) \equiv_{G'} aby$ and $\varphi_{\boldsymbol{1}}(xz) \equiv_{G'} aby\rho(x)$. If $y\ne 1$, take the former; the latter otherwise. In both cases, we have $\varphi_w(xz) = aby'z_2$ for some $w\in X$, $y'\in B_1\setminus \{1\}$ and $z_2\in G'$. 
If $k$ is the smallest integer such that $\rho^k(y')\notin B_1$, then $\varphi_{\boldsymbol{0}^k}(\varphi_{w}(xz)^{2^k})= b\rho^k(y')z'$
for some $z'\in G'$, because
\[\varphi_{\boldsymbol{0}}((aby'z_2)^2)= b\rho(y')\varphi_{\boldsymbol{1}}(z_2)a\omega(y')\varphi_{\boldsymbol{0}}(z_2)=a\omega(y')b\rho(y')z_3\]
for some $z_3 \in G'$ (noting that $\varphi_{\boldsymbol{1}}(z_2)\varphi_{\boldsymbol{0}}(z_2)\in  G'$).
If $\rho(y')\notin B_1$, then $y'\notin B_0$, which means that $\omega(y')=a$. Thus, $\varphi_{\boldsymbol{0}}((aby'z_2)^2) = b\rho(y')z_3$.
Otherwise, we have $\varphi_{\boldsymbol{0}}((aby'z_2)^2)=ab\rho(y')z_3$ with $\rho(y')\in B_1\setminus \{1\}$ and we can repeat this process.
Thus, eventually, we will obtain $\varphi_{\boldsymbol{0}^k}(\varphi_{w}(xz)^{2^k})= x'z'$, where $x'=b\rho^k(y')\in B\setminus\{1\}$.

Notice that $l(\varphi_{\boldsymbol{0}}(g^2)) \leq \frac{2l(g) +1}{2}\leq l(g)$ for all $g\in  G$ (since $l$ takes integer values). Thus, we get by induction that $l(\varphi_{\boldsymbol{0}^k}(g^{2^k})) \leq l(g)$. Therefore, $l(x'z')\leq l(\varphi_w(xz)) \leq \frac{l(xz)+1}{2}$.

Thus, in all cases, there is a vertex $v\in T$ and elements $x'\in B\setminus\{1\}$, 
$z'\in G'$ such that $x'z'$ is in the projection  $\langle xz \rangle_v$, with $l(x'z')\leq \frac{l(xz)+1}{2}$. 
By repeating this process as necessary, we can assume that $l(x'z')\leq 1$, which is equivalent to saying that $x'z'\in B$. 
Since $x'\ne 1$, we must have that $x'z'=x'\ne 1$ by Proposition \ref{prop:AbelianisationGIsAB}, so $x'$ is our desired $\beta$. 
\end{proof}

\begin{lem}\label{lemma:PowersOf(ab)AreForever}
Let $M\leq  G$ be a subgroup that contains $(ab)^k$ for some $k\in \ZZ$. 
Then, $(ab)^k\in M_v$ for all $v\in T$.
\end{lem}
\begin{proof}
Because $\psi((ab)^{2k}) = ((ba)^k, (ab)^k)$,
 the result is true for the vertices of the first level, and hence for any vertex by induction.
\end{proof}

\begin{lem}\label{lemma:HContains(b)AndAPowerOf(ab)}
Let $M\leq  G$ be a dense subgroup.
 Then there exist a vertex $v\in T$ and an odd $k\in \NN$ such that $M_v$ contains $b$ and $(ab)^k$.
\end{lem}
\begin{proof}
It follows from Lemma \ref{lemma:HContainsPowerOf(ab)} that there exist a vertex $w\in T$ and an odd $k\in \NN$
such that $(ab)^k\in M_w$. 
By Lemma \ref{lem:MdenseActsLikeG}, $M_w$ is  dense in $ G$ because $M$ is dense in $G$.
Lemma \ref{lemma:HContainsElementOfB} then implies that there exist a vertex $w'\in T$ and an element $\beta\in B\setminus \{1\}$ such that $\beta\in (M_{w})_{w'} = M_{ww'}$.
Since $\varphi_{\boldsymbol{1}}(\beta) = \rho(\beta)$, we see by induction that there exist some $l\in \NN$ and $x\in B\setminus B_0$ such that $x\in M_{u}$, where $u=ww'\boldsymbol{1}^{l}$.

 As $(ab)^k \in M_w$, Lemma \ref{lemma:PowersOf(ab)AreForever}, guarantees that $(ab)^k\in M_{u}$. Then 
$$\varphi_{\boldsymbol{1}}((ab)^{k}x(ab)^k)=(ab)^{\frac{k-1}{2}}aab(ab)^{\frac{k-1}{2}} 
=(ab)^{\frac{k-1}{2}}(ba)^{\frac{k-1}{2}}b 
=b.$$
Hence, $b, (ab)^k \in M_v$, where $v=u\boldsymbol{1}$.
\end{proof}

\begin{prop}\label{prop:DenseSubgroupsProjectToSomeHq}
Let $M < G$ be a proper dense subgroup. 
Then, there exist a vertex $v\in T$ and an odd $l\in \NN$ such that $M_v = \langle (ab)^l, B \rangle$.
\end{prop}
\begin{proof}
According to Lemma \ref{lemma:HContains(b)AndAPowerOf(ab)}, there exist a vertex $w\in T$ and an odd $k\in \NN$ such that $b,(ab)^k \in M_w$.

Let us write $B=\{1, \beta_1, \beta_2, \dots \beta_{2^{m}-1}\}$. 
 By Lemma \ref{lem:MdenseActsLikeG}, $M_w$ is dense in $G$, as $M$ is dense in $G$.
  Therefore, for every $\beta_i\in B$, there exists $z_i\in G'$ such that $\beta_i az_i\in M_w$.
   This implies that $\rho(\beta_i )a\Theta(z_i)\in M_{w\boldsymbol{1}}$ for all $1\leq i \leq 2^m-1$
   because
   $$\rho(\beta_i )a\Theta(z_i)=\rho(\beta_i)\varphi_0(z_i)a\varphi_1(z_i)=\begin{cases}
   \varphi_{\boldsymbol{1}}((\beta_iaz_i)^2) & \text{ if } \beta_i\notin\ker\omega \\
   \varphi_{\boldsymbol{1}}(\beta_iaz_ib\beta_iaz_i) & \text{ if } \beta_i\in\ker\omega.
   \end{cases}$$
  %
 %
Since $b\in M_{w\boldsymbol{1}}$, and $|\Theta(z)|\leq |z|$ for all $z\in G'$, we can repeat the above procedure until we reach some $N\in\NN$ such that $|\Theta^n(z_i)|=|\Theta^N(z_i)|$ for all $n\geq N$ and all $i\in\{1,\dots,2^m-1\}$.
Lemma \ref{lemma:ThetaMapConvergence} then yields, for each $i\in\{1,\dots,2^m-1\}$, elements $x_i\in B$ and $l_i\in\NN$ such that $\beta_iaz'_i\in M_{w\boldsymbol{1}^N}$ where $z'_i = ax_ia(ba)^{2l_i}x_i$ or $x_iax_ia$. 
Since $$\psi(z'_i) = (\rho(x_i)(ab)^{l}\omega(x_i), \omega(x_i)(ba)^l\rho(x_i) )  \text{ or } (\omega(x_i)\rho(x_i), \rho(x_i)\omega(x_i) ),$$
we have 
\begin{align*}
\rho(\beta_i) &=\rho(\beta_i)\rho(x_i)(ab)^l\omega(x_i)\omega(x_i)(ba)^l\rho(x_i) \text{ or } \rho(\beta_i)\omega(x_i)\rho(x_i)\rho(x_i)\omega(x_i)\\
&=\rho(\beta_i)\varphi_{\boldsymbol{0}}(z'_i)\varphi_{\boldsymbol{1}}(z'_i) 
=\begin{cases}
\varphi_{\boldsymbol{1}}((\beta_iaz'_i)^2) & \text{ if } \beta_i\in\ker\omega\\
\varphi_{\boldsymbol{1}}(\beta_iaz'_ib\beta_iaz'_i) & \text{ if } \beta_i\notin\ker\omega.
\end{cases}
\end{align*}
Thus $B \leq M_{u}$ where $u=w\boldsymbol{1}^{N+1}$.
Moreover, $(ab)^k\in M_u$ by Lemma \ref{lemma:PowersOf(ab)AreForever}. 
 Since $M$ is proper and dense, $M_{u}$ is also proper and dense, by Proposition \ref{prop:PervovaProperProjection}. 
Lemma \ref{lem:HqcontainsDivisors} then implies that there exist a vertex $v\in T$ and an odd $l\in \NN$ such that $M_v=\langle (ab)^l, B \rangle$.
\end{proof}

\begin{lem}\label{lem:maximalprojectstoHq}
	Let $M<G$ be a maximal subgroup of infinite index. 
	Then there exist $v\in T$ and an odd prime $q\in\NN$ such that $M_v=\langle(ab)^q, B\rangle =H(q)$.	
\end{lem}
\begin{proof}
	Since a maximal subgroup of infinite index is proper and dense, by Proposition \ref{prop:DenseSubgroupsProjectToSomeHq} there exist $v\in T$ and $l\in\NN$ odd such that $M_v=\langle (ab)^l, B\rangle$. 
	So it suffices to show that if $l$ is not prime then $M$ is not maximal. 
	
	Suppose that $l$ is not prime, so that there is an odd prime $q$ dividing $l$. 
	We first show that $(ab)^4\in K=\langle [a,\beta] \mid \beta \in B_1\rangle ^G$. 
	For this, recall that there exists $d\in B_0\setminus B_1$ such that $(ad)^2=1$ and that, since $b\notin B_1$ and $|B:B_1|=2$, there is $f\in B_1$ with $b=fd$. 
	Then $(ab)^2=[a,b]=[a,fd]= [a,d][a,f]^d$ with $[a,f]^d\in K$
	so that $(ab)^4=[a,b]^2\equiv_K [a,d]^2=(ad)^4=1$, as required. 
	
	Since $G$ is regular branch over $K$, we can find, by the above claim, some $g\in \rist_G(v)$ such that $\varphi_v(g)=(ab)^{4q}$. 
	Note that $g\notin M$ as  $(ab)^{4q}\notin H(l)=M_v$. 
	Thus $M$ is strictly contained in $L:=\langle M, g\rangle$. 
	We will show that $L:=\langle M, g\rangle $ is a proper subgroup of $G$ by showing that $L_v$ is a proper subgroup of $G$. 
	Each $\gamma\in \St_L(v)$ can be written as
	$$\gamma=m_1^{-1}g^{i_1}m_1m_2^{-1}g^{i_2}m_2\cdots m_k^{-1}g^{i_k}m_km$$
	for some $k\in\NN$ where $m_j,m\in M$ and $i_j\in\ZZ$ for $1\leq j\leq k$. 
	Since $g\in \St_G(n)$, so is $m_j^{-1}g^{i_j}m_j\in\St_G(n)\leq\St_G(v)$ for $1\leq j\leq k$; thus $m\in\St_G(v)$. 
	Because $g\in \rist_G(v)$, if $m_j\notin \St_G(v)$ then $\varphi_v(m_j^{-1}g^{i_j}m_j)=1$ and so 
	\begin{align*}
	\varphi_v(\gamma)&=\varphi_v(m_1^{-1}g^{i_1}m_1)\varphi_v(m_2^{-1}g^{i_2}m_2)\cdots\varphi_v(m_k^{-1}g^{i_k}m_k)\varphi_v(m)\\
	&\in \langle M_v, \varphi_v(g)\rangle=  \langle (ab)^l, B, (ab)^{4q}\rangle=H(q)
	\end{align*}
	where the last equality holds because $q$ is the greatest common divisor of $l$ and $4q$.
	Therefore $L_v \leq H(q) \lneq G$, which implies that $L\lneq G$, as $G$ is self-replicating. 
\end{proof}

\bibliographystyle{abbrv}
\bibliography{maximal_sunic}

\end{document}